\documentclass[reqno]{amsart}

\usepackage{mathrsfs}
\usepackage{amsmath}
\usepackage{amssymb}
\usepackage{cite}
\usepackage{comment}
\usepackage{latexsym}
\usepackage{graphicx}

\usepackage{color}

\theoremstyle{plain}
\newtheorem{thm}{Theorem}[section]
\newtheorem{lem}[thm]{Lemma}

\newtheorem{dfn}[thm]{Definition}

\newtheorem{rmk}[thm]{Remark}

\def\D{\mathrm{D}}
\def\G{\mathscr{G}}

\def\S{\mathscr{S}}
\def\T{\mathscr{T}}

\def\d{\mathrm{d}}

\def\Cset{\mathbb{C}}
\def\Kset{\mathbb{K}}
\def\Lset{\mathbb{L}}
\def\Nset{\mathbb{N}}
\def\Qset{\mathbb{Q}}
\def\Rset{\mathbb{R}}
\def\Sset{\mathbb{S}}
\def\Tset{\mathbb{T}}
\def\Zset{\mathbb{Z}}

\def\id{\mathrm{id}}

\def\epsilon{\varepsilon}

\def\GL{\mathrm{GL}}

\def\epsilon{\varepsilon}

\DeclareMathOperator{\arccot}{arccot}
\DeclareMathOperator{\arccosh}{arccosh}
\DeclareMathOperator{\arctanh}{arctanh}
\DeclareMathOperator{\im}{Im}

\makeatletter
 \@addtoreset{equation}{section}
\makeatother
\def\theequation{\arabic{section}.\arabic{equation}}

\begin{document}


\title[Nonintegrability of the restricted three-body problem]%
{Nonintegrability of the restricted three-body problem}

\author{Kazuyuki Yagasaki}

\address{Department of Applied Mathematics and Physics, Graduate School of Informatics,
Kyoto University, Yoshida-Honmachi, Sakyo-ku, Kyoto 606-8501, JAPAN}
\email{yagasaki@amp.i.kyoto-u.ac.jp}

\date{\today}

\begin{abstract}
The problem of nonintegrability of the circular restricted three-body problem
 is very classical and important in the theory of dynamical systems.
It was partially solved by Poincar\'e in the nineteenth century:
He showed that there exists no real-analytic first integral
 which depends analytically on the mass ratio of the second body to the total
 and is functionally independent of the Hamiltonian.
When the mass of the second body becomes zero,
 the restricted three-body problem reduces to the two-body Kepler problem.
We prove the nonintegrability of the restricted three-body problem
 both in the planar and spatial cases for any nonzero mass of the second body.
Our basic tool of the proofs is a technique developed here
 for determining whether perturbations of integrable systems which may be non-Hamiltonian
 are not meromorphically integrable near resonant periodic orbits
 such that the first integrals and commutative vector fields also depend meromorphically
 on the perturbation parameter.
The technique is based on generalized versions due to  Ayoul and Zung
 of the Morales-Ramis and Morales-Ramis-Sim\'o theories.
We emphasize that our results are not just applications of the theories.
 \end{abstract}
\maketitle


\section{Introduction}

\begin{figure}[t]
\includegraphics[scale=0.8]{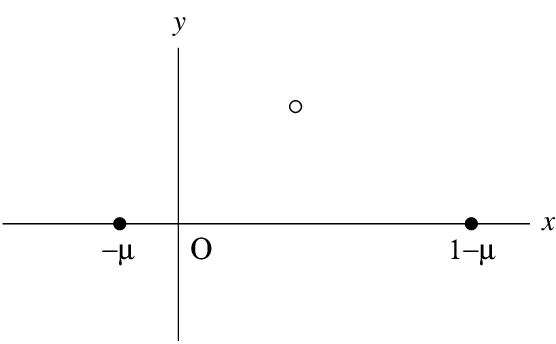}
\caption{Configuration of the circular restricted three-body problem in the rotational frame.
\label{fig:1a}}
\end{figure}

In this paper we study the nonintegrability of the circular restricted three-body problem 
 for the planar case,
\begin{equation}
\begin{split}
&
\dot{x}=p_x+y,\quad
\dot{p}_x=p_y+\frac{\partial U_2}{\partial x}(x,y),\\
&
\dot{y}=p_y-x,\quad
\dot{p}_y=-p_x+\frac{\partial U_2}{\partial y}(x,y),
\end{split}
\label{eqn:pp}
\end{equation}
and for the spatial case,
\begin{equation}
\begin{split}
&
\dot{x}=p_x+y,\quad
\dot{p}_x=p_y+\frac{\partial U_3}{\partial x}(x,y,z),\\
&
\dot{y}=p_y-x,\quad
\dot{p}_y=-p_x+\frac{\partial U_3}{\partial y}(x,y,z),\\
&
\dot{z}=p_z,\quad
\dot{p}_z=\frac{\partial U_3}{\partial z}(x,y,z),
\end{split}
\label{eqn:sp}
\end{equation}
where
\begin{align*}
&
U_2(x,y)=\frac{\mu}{\sqrt{(x-1+\mu)^2+y^2}}+\frac{1-\mu}{\sqrt{(x+\mu)^2+y^2}},\\
&
U_3(x,y,z)=\frac{\mu}{\sqrt{(x-1+\mu)^2+y^2+z^2}}+\frac{1-\mu}{\sqrt{(x+\mu)^2+y^2+z^2}}.
\end{align*}
The systems \eqref{eqn:pp} and \eqref{eqn:sp} are Hamiltonian with the Hamiltonians
\begin{equation*}
H_2(x,y,p_x,p_y)=\tfrac{1}{2}(p_x^2+p_y^2)+(p_xy-p_yx)-U_2(x,y)
\end{equation*}
and
\begin{equation*}
H_3(x,y,z,p_x,p_y,p_z)=\tfrac{1}{2}(p_x^2+p_y^2+p_z^2)+(p_xy-p_yx)-U_3(x,y,z),
\end{equation*}
respectively, and represent the dimensionless equations of motion of the third massless body
  subjected to the gravitational forces from the two primary bodies with mass $\mu$ and $1-\mu$
  which remain at $(1-\mu,0)$ and $(-\mu,0)$, respectively,
  on the $xy$-plane in the rotational frame,
  under the assumption that the primaries rotate counterclockwise
  on the circles about their common center of mass at the origin
  in the inertial coordinate frame.
See Fig.~\ref{fig:1a}.
Their nonintegarbllity means that
 Eq.~\eqref{eqn:pp} (resp. Eq.~\eqref{eqn:sp}) does not have one first integral
 (resp. two first integrals)
 which is (resp. are) functionally independent of the Hamiltonian $H_2$ (resp. $H_3$).
See \cite{A89,M99} for the definition of integrability of general Hamiltonian systems,
 and, e.g., Section~4.1 of \cite{MO17}
 for more details on the derivation and physical meaning of \eqref{eqn:pp} and \eqref{eqn:sp}.

The problem of nonintegrability of \eqref{eqn:pp} and \eqref{eqn:sp}
 is very classical and important in the theory of dynamical systems.
In his famous memoir \cite{P90},
 which was related to a prize competition celebrating the 60th birthday of King Oscar II,
 Henri Poincar\'e studied the planar case
 and discussed the nonexistsnce of a first integral
 which is analytic in the state variables and parameter $\mu$ near $\mu=0$
 and functionally independent of the Hamiltonian.
His approach was improved significantly
 in the first volume of his masterpieces \cite{P92} published two years later:
 he showed the nonexistence of such a first integral for the restricted three-body problem
 in the planar case.
See \cite{B96} for an account of his work
 from mathematical and historical perspectives.
His result was also explained in \cite{AKN06,K83,K96,W37}.
Moreover, a remarkable progress has been made
 on the planar problem \eqref{eqn:pp} in a different direction recently:
Guardia et al. \cite{GMS16} showed the occurrence of transverse intersection
 between the stable and unstable manifolds of the infinity for any $\mu\in(0,1)$
 in a region far from the primaries
 in which $r=\sqrt{x^2+y^2}$ and its conjugate momentum are sufficiently large.
This implies, e.g., by Theorem~3.10 of \cite{M73}, 
 the real-analytic nonintetgrabilty of \eqref{eqn:pp}
 as well as the existence of oscillatory motions
 such that $\limsup_{t\to\infty}r(t)=\infty$ while $\liminf_{t\to\infty}r(t)<\infty$.
Similar results were obtained much earlier
 when $\mu>0$ is sufficiently small in \cite{LS80}
 and for any $\mu\in(0,1)$ except for a certain finite number of the values in \cite{X92}.
Note that these results immediately say nothing about
 the nonintegrability of the spatial problem \eqref{eqn:sp}.

Moreover, the nonintegrability of the general three-body problem
 is now well understood, in comparison with the restricted one.
Tsygvintsev \cite{T00,T01a} proved the nonintegrability of the general planar three-body problem
 near the Lagrangian parabolic orbits in which the three bodies form an equilateral triangle
 and move along certain parabolas, using Ziglin's method \cite{Z82b}.
Boucher and Weil \cite{BW03} also obtained a similar result,
 using the Morales-Ramis theory \cite{M99,MR01},
 which is considered as an extension of the Ziglin method,
 while it was proven for the case of equal masses a little earlier in \cite{B00}.
Moreover, Tsygvintsev \cite{T01b,T03,T07} proved the nonexistence
 of a single additional first integral near the Lagrangian parabolic orbits
when
\[
\frac{m_1m_2+m_2m_3+m_3m_1}{(m_1+m_2+m_3)^2}\notin\{\tfrac{1}{3},\tfrac{2}{9},\tfrac{8}{27}\},
\]
where $m_j$ represents the mass of the $j$th body for $j=1,2,3$.
Subsequently, Morales-Ruiz and Simon \cite{MS09}
 succeeded in removing the three exceptional cases and extended the result
 to the space of three or more dimensions.
Ziglin \cite{Z00} also proved the nonintegrability of the general three-body problem
 near a collinear solution which was used by Yoshida \cite{Y87}
 for the problem in the one-dimensional space much earlier, in the space of any dimension
 when two of the three masses, say $m_1,m_2$, are nearly equal
 but neither $m_3/m_1$ nor $m_3/m_2\in\{11/12,1/4,1/24\}$.
Maciejewski and Przybylska \cite{MP11}
 discussed the three-body problem with general homogeneous potentials.
It should be noted that
 Ziglin \cite{Z00} and Morales-Ruiz and Simon \cite{MS09}
 also discussed the general $N$-body problem.
We remark that these results say nothing
 about the nonintegrability of the restricted three-body problem
 obtained by limiting manipulation from the general one.
In particular, there exists no nonconstant solution corresponding to the Lagrangian parabolic orbits
 or collinear solutions in the restricted one.

Here we show the nonintegrability
 of the three-body problems \eqref{eqn:pp} and \eqref{eqn:sp}
 near the primaries for any $\mu\in(0,1)$ fixed.
To state our result precisely,
 we use the following treatment originally made in \cite{C13}.
We first introduce the new variables $u_1,u_2\in\Cset$ given by
\[
u_1^2-(x-1+\mu)^2-y^2=0,\quad
u_2^2-(x+\mu)^2-y^2=0
\]
and
\[
u_1^2-(x-1+\mu)^2-y^2-z^2=0,\quad
u_2^2-(x+\mu)^2-y^2-z^2=0
\]
and regard \eqref{eqn:pp} and \eqref{eqn:sp} as Hamiltonian systems
 on the four- and six-dimensional complex manifolds (algebraic varieties)
\begin{align*}
\S_2=&\{(x,y,p_x,p_y,u_1,u_2)\in\Cset^6\\
&\quad\mid
 u_1^2-(x-1+\mu)^2-y^2=u_2-(x+\mu)^2-y^2=0\}
\end{align*}
and
\begin{align*}
\S_3=&\{(x,y,z,p_x,p_y,p_z,u_1,u_2)\in\Cset^8\\
& \quad\mid u_1^2-(x-1+\mu)^2-y^2-z^2=u_2-(x+\mu)^2-y^2-z^2=0\},
\end{align*}
respectively.
Let $\pi_2:\S_2\to\Cset^4$ and $\pi_3:\S_3\to\Cset^6$ be the projections such that
\[
\pi_2(x,y,p_x,p_y,u_1,u_2)=(x,y,p_x,p_y)
\]
and
\[
\pi_3(x,y,z,p_x,p_y,p_z,u_1,u_2)=(x,y,z,p_x,p_y,p_z),
\]
and let 
\begin{align*}
\Sigma(\S_2)=&\{u_1=(x-1+\mu)^2+y^2=0\}\\
&\quad \cup\{u_2=(x+\mu)^2+y^2=0\}\subset\S_2,\\
\Sigma(\S_3)=&\{u_1=(x-1+\mu)^2+y^2+z^2=0\}\\
 &\quad \cup\{u_2=(x+\mu)^2+y^2+z^2=0\}\subset\S_3.
\end{align*}
Note that $\pi_2$ and $\pi_3$ are singular
 on $\Sigma(\S_2)$ and $\Sigma(\S_3)$, respectively.
The sets $\Sigma(\S_2)$ and $\Sigma(\S_3)$
 are called the \emph{critical sets} of  $\S_2$ and $\S_3$, respectively.
The systems \eqref{eqn:pp} and \eqref{eqn:sp} are, respectively, rewritten as
\begin{equation}
\begin{split}
&
\dot{x}=p_x+y,\quad
\dot{p}_x=p_y-\frac{\mu}{u_1^3}(x-1+\mu)-\frac{1-\mu}{u_2^3}(x+\mu),\\
&
\dot{y}=p_y-x,\quad
\dot{p}_y=-p_x-\frac{\mu}{u_1^3}y-\frac{1-\mu}{u_2^3}y,\\
&
\dot{u}_1=\frac{1}{u_1}((x-1+\mu)(p_x+y)+y(p_y-x)),\\
&
\dot{u}_2=\frac{1}{u_2}((x+\mu)(p_x+y)+y(p_y-x))
\end{split}
\label{eqn:ppu}
\end{equation}
and
\begin{equation}
\begin{split}
&
\dot{x}=p_x+y,\quad
\dot{p}_x=p_y-\frac{\mu}{u_1^3}(x-1+\mu)-\frac{1-\mu}{u_2^3}(x+\mu),\\
&
\dot{y}=p_y-x,\quad
\dot{p}_y=-p_x-\frac{\mu}{u_1^3}y-\frac{1-\mu}{u_2^3}y,\\
&
\dot{z}=p_z,\quad
\dot{p}_z=-\frac{\mu}{u_1^3}z-\frac{1-\mu}{u_2^3}z,\\
&
\dot{u}_1=\frac{1}{u_1}((x-1+\mu)(p_x+y)+y(p_y-x)+zp_z),\\
&
\dot{u}_2=\frac{1}{u_2}((x+\mu)(p_x+y)+y(p_y-x)+zp_z)
\end{split}
\label{eqn:spu}
\end{equation}
which are rational on $\S_2$ and $\S_3$.
We prove the following theorem.

\begin{thm}
\label{thm:main}
The circular restricted three-body problem \eqref{eqn:pp} $($resp. \eqref{eqn:sp}$)$
 does not have a complete set of first integrals in involution
 that are functionally independent almost everywhere
 and meromorphic in $(x,y,p_x,p_y,u_1,u_2)$ $($resp. in $(x,y,z,p_x,p_y,p_z,u_1,u_2))$
 except on $\Sigma(\S_2)$ $($resp. on $\Sigma(\S_3))$
 in punctured neighborhoods of
\begin{gather*}
(x,y)=(-\mu,0)\mbox{ and }(1-\mu,0)\\
(\mbox{resp.}\ (x,y,z)=(-\mu,0,0)\mbox{ and }(1-\mu,0,0))
\end{gather*}
for any $\mu\in(0,1)$, as Hamiltonian systems on $\S_2$ $($resp. on $\S_3)$.
\end{thm}

Proofs of Theorem~\ref{thm:main} are given in Section~3 for the planar case \eqref{eqn:pp}
 and in Section~4 for the spatial case \eqref{eqn:sp}.
Our basic tool of the proofs is a technique developed in Section~2 for
\begin{equation}
\dot{I}=\epsilon h(I,\theta;\epsilon),\quad
\dot{\theta}=\omega(I)+\epsilon g(I,\theta;\epsilon),\quad
(I,\theta)\in\Rset^\ell\times\Tset^m,
\label{eqn:aasys}
\end{equation}
where $\ell,m\in\Nset$, $\Tset^m=(\Rset/2\pi\Zset)^m$, 
 $\epsilon$ is a small parameter such that $0<|\epsilon|\ll 1$,
 and $\omega:\Rset^\ell\to\Rset^m$,
 $h:\Rset^\ell\times\Tset^m\times\Rset\to\Rset^\ell$
 and $g:\Rset^\ell\times\Tset^m\times\Rset\to\Rset^m$
 are meromorphic in their arguments.
The system~\eqref{eqn:aasys} is Hamiltonian if $\epsilon=0$ or
\[
\D_I h(I,\theta;\epsilon)\equiv-\D_\theta g(I,\theta;\epsilon)
\]
as well as $\ell=m$, and non-Hamiltonian if not.
The developed technique enables us to determine
 whether the system \eqref{eqn:aasys} is not meromorphically integrable
 in the Bogoyavlenskij sense \cite{B98} (see Definition~\ref{dfn:1a} below)
 such that the first integrals and commutative vector fields
 also depend meromorphically on $\epsilon$ near $\epsilon=0$,
 like the result of Poincar\'e \cite{P90,P92} stated above,
 when the domains of the independent and dependent variables are extended to regions
 in $\Cset$ and $\Cset^\ell\times(\Cset/2\pi\Zset)^m$, respectively.
The general definition of integrability adopted here is precisely stated as follows.

\begin{dfn}[Bogoyavlenskij]
\label{dfn:1a}
For $n\in\Nset$ an $n$-dimensional dynamical system
\[
\dot{x}=f(x),\quad\mbox{$x\in\Rset^n$ or $\Cset^n$},
\]
is called \emph{$(q,n-q)$-integrable} or simply \emph{integrable} 
 if there exist $q$ vector fields $f_1(x)(:= f(x)),f_2(x),\dots,f_q(x)$
 and $n-q$ scalar-valued functions $F_1(x),\dots,F_{n-q}(x)$ such that
 the following two conditions hold$\,:$
\begin{enumerate}
\setlength{\leftskip}{-1.8em}
\item[\rm(i)]
$f_1(x),\dots,f_q(x)$ are linearly independent almost everywhere
 and commute with each other,
 i.e., $[f_j,f_k](x):=\D f_k(x)f_j(x)-\D f_j(x)f_k(x)\equiv 0$ for $j,k=1,\ldots,q$,
 where $[\cdot,\cdot]$ denotes the Lie bracket$\,;$
\item[\rm(ii)]
The derivatives $\D F_1(x),\dots, \D F_{n-q}(x)$ are linearly independent almost everywhere
 and $F_1(x),\dots,F_{n-q}(x)$ are first integrals of $f_1, \dots,f_q$,
 i.e., $\D F_k(x)\cdot f_j(x)\equiv 0$ for $j=1,\ldots,q$ and $k=1,\ldots,n-q$,
 where ``$\cdot$'' represents the inner product.
\end{enumerate}
We say that the system is \emph{meromorphically} 
 \emph{integrable}
 if the first integrals and commutative vector fields are meromorphic. 
\end{dfn}

Definition~\ref{dfn:1a} is considered as a generalization of 
 Liouville-integrability for Hamiltonian systems \cite{A89,M99}
 since an $n$-degree-of-freedom Liouville-integrable Hamiltonian system with $n\ge 1$
 has not only $n$ functionally independent first integrals
 but also $n$ linearly independent commutative (Hamiltonian) vector fields
 generated by the first integrals.
When $\epsilon=0$,
 the system~\eqref{eqn:aasys} is meromorphically $(m,\ell)$-integrable
 in the Bogoyavlenskij sense:
 $F_j(I,\theta)=I_j$, $j=1,\ldots,\ell$, are first integrals
 and $f_j(I,\theta)=(0,e_j)\in\Rset^\ell\times\Rset^m$, $j=2,\ldots,m$,
 give $m$ commutative vector fields along with its own vector field,
 where $e_j$ is the $m$-dimensional vector of which the $j$th element is the unit
 and the other ones are zero. 
Conversely, a general $(m,\ell)$-integrable system is transformed
 to the form \eqref{eqn:aasys} with $\epsilon=0$
 if the level set of the first integrals $F_1(x),\ldots,F_m(x)$
 has a connected compact component.
See \cite{B98,Z18} for more details.
Thus, the system \eqref{eqn:aasys} can be regarded
 as a normal form for perturbations of general $(m,\ell)$-integrable systems.

Systems of the form \eqref{eqn:aasys} have attracted much attention,
 especially when they are Hamiltonian.
See \cite{A89,AKN06,K96} and references therein for more details.
In particular, Kozlov \cite{K96} extended the famous result of Poincar\'e \cite{P90,P92}
 for Hamiltonian systems to the general analytic case of \eqref{eqn:aasys}
 and gave sufficient conditions for nonexistence of additional real-analytic first integrals
 depending analytically on $\epsilon$ near $\epsilon=0$.
See also \cite{AKN06,K83} for his result in Hamiltonian systems. 
Moreover, Motonaga and Yagasaki \cite{MY21a} gave sufficient conditions
 for real-analytic nonintegrability of general nearly integrable systems
 in the Bogoyavlenskij sense
 such that the first integrals and commutative vector fields also depend real-analytically
 on $\epsilon$ near $\epsilon=0$.
The technique developed in Section~2 is different from them
 and based on generalized versions due to Ayoul and Zung \cite{AZ10}
 of the Morales-Ramis and Morales-Ramis-Sim\'o theories \cite{M99,MR01,MRS07}.
See Appendix~A of \cite{Y21a}
 for a brief review of the previous results and their comparison with the developed technique.
Our technique can also be applied to several nearly integrable systems
 containing time-periodic perturbation of single-degree-of-freedom Hamiltonian systems
 such as the periodically forced Duffing oscillator and pendulum \cite{MY21b,Y21a}.
Moreover, it can be used directly
 to give a new proof of Poincar\'e's result of \cite{P92}
 on the restricted three-body problem \cite{Y21b}.
The systems \eqref{eqn:pp} and \eqref{eqn:sp} are transformed
 to the form \eqref{eqn:aasys} in the punctured neighborhoods in Theorem~\ref{thm:main}
 and the technique is applied to them.
We emphasize that our results are not just applications of 
 the Morales-Ramis and Morales-Ramis-Sim\'o theories
 or their generalized versions.

\section{Determination of nonintegrability for \eqref{eqn:aasys}}

In this section, we give a technique for determining
 whether the system \eqref{eqn:aasys} is not meromorphically Bogoyavlenskij-integrable
 such that the first integrals and commutative vector fields
 also depend meromorphically on $\epsilon$ near $\epsilon=0$.

\begin{figure}
\includegraphics[scale=0.8]{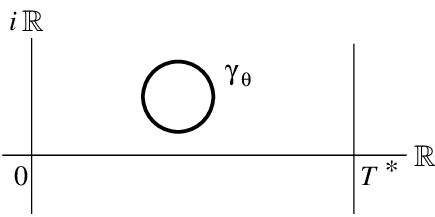}
\caption{Assumption~(A2).\label{fig:2a}}
\end{figure}

When $\epsilon=0$, Eq.~\eqref{eqn:aasys} becomes
\begin{equation}
\dot{I}=0,\quad
\dot{\theta}=\omega(I).
\label{eqn:aasys0}
\end{equation}
We assume the following on the unperturbed system \eqref{eqn:aasys0}:
\begin{enumerate}
\setlength{\leftskip}{-1em}
\item[\bf(A1)]
For some $I^\ast\in\Rset^\ell$, a resonance of multiplicity $m-1$,
\[
\dim_\Qset\langle\omega_1(I^\ast),\ldots,\omega_m(I^\ast)\rangle=1,
\]
occurs with $\omega(I^\ast)\neq 0$,
 i.e., there exists a constant $\omega^\ast>0$ such that 
\[
\frac{\omega(I^\ast)}{\omega^\ast}\in\Zset^m\setminus\{0\},
\]
where $\omega_j(I)$ is the $j$th element of $\omega(I)$ for $j=1,\ldots,m$.
\end{enumerate}
Note that we can replace $\omega^\ast$ with $\omega^\ast/n$ for any $n\in\Nset$ in (A1).
We refer to the $m$-dimensional torus $\T^\ast=\{(I^\ast,\theta)\mid\theta\in\Tset^m\}$
 as the \emph{resonant torus}
 and to periodic orbits $(I,\theta)=(I^\ast,\omega(I^\ast)t+\theta_0)$, $\theta_0\in\Tset^m$,
 on $\T^\ast$ as the \emph{resonant periodic orbits}.
Let $T^\ast=2\pi/\omega^\ast$.
We also make the following assumption.
\begin{enumerate}
\setlength{\leftskip}{-1em}
\item[\bf(A2)]
For some $k\in\Zset_{\ge 0}:=\Nset\cup\{0\}$ and $\theta\in\Tset^m$ there exists a closed loop $\gamma_\theta$
 in a region including $(0,T^\ast)\subset\Rset$ in $\Cset$ such that
 $\gamma_\theta\cap(i\Rset\cup(T^\ast+i\Rset))=\emptyset$ and
\begin{equation}
\mathscr{I}^k(\theta):=\frac{1}{k!}\D\omega(I^\ast)\int_{\gamma_\theta}
 \D_\epsilon^k h(I^\ast,\omega(I^\ast)\tau+\theta;0)\d\tau
\label{eqn:A2}
\end{equation}
is not zero.
See Fig.~\ref{fig:2a}
\end{enumerate}
Note that the condition $\gamma_\theta\cap(i\Rset\cup(T^\ast+i\Rset))=\emptyset$
 is not essential in (A2), since it always holds
 by replacing $\omega^\ast$ with $\omega^\ast/n$ for sufficiently large $n\in\Nset$ if necessary.
We prove the following theorem
 which guarantees that conditions~(A1) and (A2) are sufficient for nonintegrability of \eqref{eqn:aasys}
 in a certain meaning.

\begin{thm}
\label{thm:tool}
Let $\Gamma$ be any domain in $\Cset/T^\ast\Zset$
 containing $\Rset/T^\ast\Zset$ and $\gamma_\theta$.
Suppose that assumption~{\rm(A1)} and {\rm(A2)} hold for some $\theta=\theta_0\in\Tset^m$.
Then the system \eqref{eqn:aasys} is not meromorphically integrable in the Bogoyavlenskij sense
 near the resonant periodic orbit $(I,\theta)=(I^\ast,\omega(I^\ast)\tau+\theta_0)$
 with $\tau\in\Gamma$
 such that the first integrals and commutative vector fields also depend meromorphically
 on $\epsilon$ near $\epsilon=0$,
 when the domains of the independent and dependent variables are extended to regions
 in $\Cset$ and $\Cset^\ell\times(\Cset/2\pi\Zset)^m$, respectively. 
Moreover, if {\rm(A2)} holds for any $\theta\in\Delta$, where $\Delta$ is a dense set in $\Tset^m$,
 then the conclusion holds for any resonant periodic orbit on the resonant torus $\T^\ast$.
\end{thm}

Our basic idea of the proof of Theorem~\ref{thm:tool}
 is similar to that of Morales-Ruiz \cite{M02},
 who studied time-periodic {\color{black}Hamiltonian} perturbations
 of single-degree-of-freedom Hamiltonian systems
 and showed a relationship of their nonintegrability
 with a version due to Ziglin \cite{Z82a} of the Melnikov method \cite{M63} 
 when the small parameter $\epsilon$ is regarded as a state variable.
Here the Melnikov method enables us
 to detect transversal self-intersection of complex separatrices of periodic orbits
 unlike the standard version \cite{GH83,M63,W03}.
More concretely, under some restrictive conditions,
 he essentially proved that they are meromorphically nonintegrable
 when the small parameter is taken as one of the state variables
 if the Melnikov functions are not identically zero,
 based on a generalized version due to Ayoul and Zung \cite{AZ10}
 of the Morales-Ramis theory \cite{M99,MR01}.
We also use their generalized versions of the Morales-Ramis theory and its extension,
 the Morales-Ramis-Sim\'o theory \cite{MRS07}, to  prove Theorem~\ref{thm:tool}.
These generalized theories enable us to show the nonintegrability of general differential equations
in the Bogayavlenskij sense by using differential Galois groups \cite{CH11,PS03}
 of their variational or higher-order variational equations along nonconstant particular solutions.
We extend the idea of Morales-Ruiz \cite{M02}
 to higher-dimensional non-Hamiltonian systems near periodic orbits.

For the proof of Theorem~\ref{thm:tool},
 we first consider systems of the general form
\begin{equation}
\dot{x}=f(x;\epsilon),\quad
x\in\Cset^n,
\label{eqn:fsys}
\end{equation}
where $f:\Cset^n\times\Cset\to\Cset^n$ is meromorphic,
 and describe direct consequences of the generalized versions due to Ayoul and Zung \cite{AZ10}
 of the Morales-Ramis theory \cite{M99,MR01}
 when the parameter $\epsilon$ is regarded as a state variable in \eqref{eqn:fsys} near $\epsilon=0$.
Let $x=\bar{x}(t)$ be a periodic orbit in the unperturbed system
\begin{equation*}
\dot{x}=f(x;0).
\end{equation*}
Taking $\epsilon$ as another state variable, we extend \eqref{eqn:fsys} as
\begin{equation}
\dot{x}=f(x;\epsilon),\quad
\dot{\epsilon}=0,
\label{eqn:efsys}
\end{equation}
in which $(x,\epsilon)=(\bar{x}(t),0)$ is a periodic orbit.
The variational equation (VE) of \eqref{eqn:efsys}
 along the periodic solution $(\bar{x}(t),0)$ is given by
\begin{equation}
\dot{y}=\D_x f(\bar{x}(t);0)y+\D_\epsilon f(\bar{x}(t);0)\lambda,\quad
\dot{\lambda}=0.
\label{eqn:gve1}
\end{equation}
We regard \eqref{eqn:gve1}
 as a linear differential equation on a Riemann surface.
Applying the version due to Ayoul and Zung \cite{AZ10} of the Morales-Ramis theory \cite{M99,MR01}
 to \eqref{eqn:efsys}, we obtain the following result.
\begin{thm}
\label{thm:MR}
If the system \eqref{eqn:fsys} is meromorphically integrable
 in the Bogoyavlenskij sense near $x=\bar{x}(t)$
 such that the first integrals and commutative vector fields
 also depend meromorphically on $\epsilon$ near $\epsilon=0$,
 then the identity component of the differential Galois group of \eqref{eqn:gve1} is commutative.
\end{thm}
See Appendix~A for necessary information on the differential Galois theory.

We can obtain a more general result for \eqref{eqn:efsys} as follows.
For simplicity, we assume that $n=1$. The general case can be treated similarly.
Letting
\[
\bar{f}^{(j,l)}=\D_x^j\D_\epsilon^l f(\bar{x}(t);0),
\]
we express the Taylor series of $f(x;\epsilon)$ about $(x,\epsilon)=(\bar{x}(t),0)$ as
\[
f(x;\epsilon)=\sum_{l=0}^\infty\sum_{j=0}^\infty
 \frac{1}{j! l!}\bar{f}^{(j,l)}(x-\bar{x}(t))^j\epsilon^l.
\]
Let
\[
x=\bar{x}(t)+\sum_{j=1}\epsilon^j y^{(j)}.
\]
Using these expressions, we write the $k$th-order VE
 of \eqref{eqn:efsys} along the periodic orbit $(x,\epsilon)=(\bar{x}(t),0)$ as
\begin{equation}
\begin{split}
\dot{y}^{(1)}=&\bar{f}^{(1,0)}y^{(1)}+\bar{f}^{(0,1)}\lambda^{(1)},\quad
\dot{\lambda}^{(1)}=0,\\
\dot{y}^{(2)}=&\bar{f}^{(1,0)}y^{(2)}+\bar{f}^{(0,1)}\lambda^{(2)}\\
& +\bar{f}^{(2,0)}(y^{(1)})^2+2\bar{f}^{(1,1)}(y^{(1)},\lambda^{(1)})
 +\bar{f}^{(0,2)}(\lambda^{(1)})^2,\quad
\dot{\lambda}^{(2)}=0,\\
\dot{y}^{(3)}=&\bar{f}^{(1,0)}(y^{(3)})+\bar{f}^{(0,1)}(\lambda^{(3)})
 +2\bar{f}^{(2,0)}(y^{(1)},y^{(2)})\\
 &+2\bar{f}^{(1,1)}(y^{(1)},\lambda^{(2)})+2\bar{f}^{(1,1)}(y^{(1)},\lambda^{(2)})
 +2\bar{f}^{(0,2)}(\lambda^{(1)},\lambda^{(2)})\\
& +\bar{f}^{(3,0)}(y^{(1)})^3+3\bar{f}^{(2,1)}((y^{(1)})^2,\lambda^{(1)})+3\bar{f}^{(1,2)}(y^{(1)},(\lambda^{(1)})^2)\\
& +\bar{f}^{(0,3)}(\lambda^{(1)})^3,\quad
\dot{\lambda}^{(3)}=0,\\
& \vdots\\
\dot{y}^{(k)}=&\sum\frac{(j+l)!}{j_1!\cdots j_r! l_1!\cdots l_s!}\\
& \quad\times\bar{f}^{(j,l)}
 ((y^{(i_1)})^{j_1},\ldots,(y^{(i_r)})^{j_r},(\lambda^{(i_1')})^{l_1},\ldots,(\lambda^{(i_s')})^{l_s}),\quad
 \dot{\lambda}^{(k)}=0, 
\end{split}
\label{eqn:gevk}
\end{equation}
where such terms as $(y^{(1)})^0,(\lambda^{(1)})^0=1$ have been substituted
 and the summation in the last equation has been taken over all integers
\[
j,l,r,s,i_1,\ldots,i_r,i_1',\ldots,i_s',j_1,\ldots,j_r,l_1,\ldots,l_s
\]
such that
\begin{align*}
&
1\le j+l\le k,\quad
i_1<i_2<\ldots<i_r,\quad
i_1'<i_2'<\ldots<i_s',\\
&
\sum_{r'=1}^r j_{r'}=j,\quad
\sum_{s'=1}^s l_{s'}=l,\quad
\sum_{r'=1}^r j_{r'}i_{r'}+\sum_{s'=1}^s l_{s'}i_{s'}=k.
\end{align*}
See \cite{MRS07} for the details on derivation of higher-order VEs in a general setting.
Substituting $y^{(j)}=0$, $j=1,\ldots,k-1$, and $\lambda^{(l)}=0$, $l=2,\ldots,k$, into \eqref{eqn:gevk},
 we obtain
\[
\dot{\lambda}^{(1)}=0,\quad
\dot{y}^{(k)}=\bar{f}^{(1,0)}(y^{(k)})+\bar{f}^{(0,k)}(\lambda^{(1)})^k,
\]
which is equivalent to
\begin{equation}
\dot{y}=\D_x f(\bar{x}(t);0)y+\frac{1}{k!}\D_\epsilon^k f(\bar{x}(t);0)\lambda,\quad
\dot{\lambda}=0
\label{eqn:grve}
\end{equation}
with $y=y^{(k)}$ and $\lambda=k!(\lambda^{(1)})^k$.
Equation \eqref{eqn:grve} is derived for $n>1$ similarly.
We regard \eqref{eqn:grve}
 as a linear differential equation on a Riemann surface, again.
Such a reduction of higher-order VEs was used for planar systems in \cite{ALMP18,AY20}.
We call \eqref{eqn:grve} the \emph{$k$th-order reduced variational equation} (RVE)
 of \eqref{eqn:efsys} around the periodic orbit $(x,\epsilon)=(\bar{x}(t),0)$.
Using the version due to Ayoul and Zung \cite{AZ10}
 of the Morales-Ramis-Sim\'o theory \cite{MRS07},
 we obtain the following result.

\begin{thm}
\label{thm:MRS}
If the system \eqref{eqn:fsys} is meromorphically integrable
 in the Bogoyavlenskij sense near $x=\bar{x}(t)$
 such that the first integrals and commutative vector fields
 also depend meromorphically on $\epsilon$ near $\epsilon=0$,
 then the identity component of the differential Galois group of \eqref{eqn:grve} is commutative.
\end{thm}

\begin{rmk}
The statement of Theorem~$\ref{thm:MRS}$ is very weak,
 compared with the original one of {\rm\cite{MRS07}},
 since the RVE \eqref{eqn:grve} is much smaller than the full higher-order VE for \eqref{eqn:fsys}.
However, it is tractable and enough for our purpose.
\end{rmk}

\begin{figure}
\includegraphics[scale=0.7]{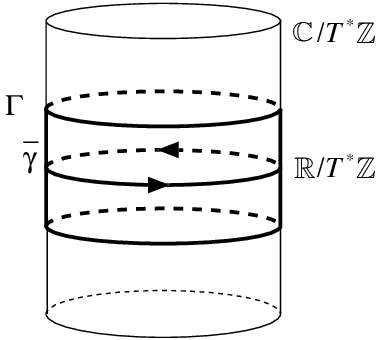}
\caption{Riemann surface $\Gamma$.
The monodromy matrix $M_\gamma$ is computed along the loop $\gamma$.
\label{fig:2b}}
\end{figure}

We return to the system \eqref{eqn:aasys} and regard $\epsilon$ as a state variable
 to rewrite it as
\begin{equation}
\dot{I}=\epsilon h(I,\theta;\epsilon),\quad
\dot{\theta}=\omega(I)+\epsilon g(I,\theta;\epsilon),\quad
\dot{\epsilon}=0,
\label{eqn:aasyse}
\end{equation}
like \eqref{eqn:efsys}.
We extend the domain of the independent variable $t$ to a region including $\Rset$ in $\Cset$,
 as stated in Theorem~\ref{thm:tool}.
The $(k+1)$th-order RVE of \eqref{eqn:aasyse}
 along the periodic orbit $(I,\theta,\epsilon)=(I^\ast,\omega(I^\ast)t+\theta_0,0)$ is given by
\begin{equation}
\begin{split}
&
\dot{\xi}=h^k(I^\ast,\omega^\ast t+\theta_0;0)\lambda,\\
&
\dot{\eta}=\D\omega(I^\ast)\xi+g^k(I^\ast,\omega^\ast t+\theta_0;0)\lambda,\\
&
\dot{\lambda}=0,
\end{split}
\quad
(\xi,\eta,\chi)\in\Cset^\ell\times\Cset^m\times\Cset.
\label{eqn:rve}
\end{equation}
where
\[
h^k(I,\theta)=\frac{1}{k!}\D_\epsilon^k h(I,\theta;0),\quad
g^k(I,\theta)=\frac{1}{k!}\D_\epsilon^k g(I,\theta;0).
\]
As a Riemann surface, we take any region $\Gamma$ in $\Cset/T^\ast\Zset$
 such that the closed loop $\gamma_\theta$ in assumption~(A2)
 as well as $\Rset/T^\ast\Zset$ is contained in $\Gamma$, as in Theorem~\ref{thm:tool}.
See Fig.~\ref{fig:2b}.
Let $\Kset_\theta\neq\Cset$ be a differential field that consists of $T^\ast$-periodic functions
 and contains the elements of $h^k(I^\ast,\omega(I^\ast)t+\theta)$
 and $g^k(I^\ast,\omega(I^\ast)t+\theta)$ with $t\in\Gamma$.
We regard the $(k+1)$th-order RVE \eqref{eqn:rve} as a linear differential equation
 over $\Kset_\theta$ on the Riemann surface $\Gamma$.
We obtain a fundamental matrix of \eqref{eqn:rve} as
\begin{equation*}
\Phi^k(t;\theta_0)=
\begin{pmatrix}
\id_\ell & 0 & \Xi^k(t;\theta_0)\\
\D\omega(I^\ast)t & \id_m & \Psi^k(t;\theta_0)\\
0 & 0 & 1
\end{pmatrix},
\end{equation*}
where $\mathrm{id}_\ell$ is the $\ell\times\ell$ identity matrix and
\begin{align*}
\Xi^k(t;\theta)
=&\int_0^t h^k(I^\ast,\omega(I^\ast)\tau+\theta)\d\tau,\\
\Psi^k(t;\theta)
=&\int_0^t\bigl(\D\omega(I^\ast)\Xi(\tau;\theta)
 +g^k(I^\ast,\omega(I^\ast)\tau+\theta)\bigr)\d\tau.
\end{align*}

Let $\G_\theta$ be the differential Galois group of \eqref{eqn:rve}
 and let $\sigma\in\G_\theta$.
Then
\[
\frac{\d}{\d t}\sigma(\Xi^k(t;\theta))=\sigma\left(\frac{\d}{\d t}\Xi^k(t;\theta)\right)
=h^k(I^\ast,\omega(I^\ast)t+\theta),
\]
so that
\begin{equation}
\sigma(\Xi^k(t;\theta))=\int_0^t h^k(I^\ast,\omega(I^\ast)\tau+\theta;0)\d\tau+C
=\Xi^k(t;\theta)+C,
\label{eqn:rmk2a}
\end{equation}
where $C$ is a constant $\ell$-dimensional vector depending on $\sigma$.
If $\Xi^k(t;\theta)\in\Kset_\theta$, then $C=0$ for any $\sigma\in\G_\theta$.
Similarly, we have
\[
\sigma(\D\omega(I^\ast)t)=\D\omega(I^\ast)t+C',
\]
where $C'$ is a constant $m\times\ell$ matrix depending on $\sigma$.
If $\D\omega(I^\ast)\neq 0$, then $C'\neq 0$ for some $\sigma\in\G_\theta$
 since $\D\omega(I^\ast)t\notin \Kset_\theta$.
However,
\begin{align*}
\frac{\d}{\d t}\sigma(\Psi^k(t;\theta))
=&\sigma\left(\frac{\d}{\d t}\Psi^k(t;\theta)\right)
=\sigma\bigl(\D\omega(I^\ast)\Xi^k(t;\theta)
+g^k(I^\ast,\omega(I^\ast)t+\theta)\bigr)\\
=&\D\omega(I^\ast)\Xi^k(t;\theta)+g^k(I^\ast,\omega(I^\ast)t+\theta)
 +\D\omega(I^\ast)C.
\end{align*}
Hence,
\[
\sigma(\Psi^k(t;\theta))=\Psi^k(t;\theta)+\D\omega(I^\ast)Ct+C'',
\]
where $C''$ is a constant $m$-dimensional vector depending on $\sigma$.
If $\Xi^k(t;\theta),\Psi^k(t;\theta)\in\Kset_\theta$,
 then $C''=0$ for any $\sigma\in\G_\theta$.
Thus, we see that
\[
\G_\theta\subset\tilde{\G}:=\{
M(C_1,C_2,C_3)\mid C_1\in\Cset^\ell,C_2\in\Cset^m,C_3\in\Cset^{m\times\ell}\},
\]
where
\[
M(C_1,C_2,C_3)=
\begin{pmatrix}
\id_\ell & 0 & C_1\\
C_3 & \id_m & C_2\\
0 & 0 & 1
\end{pmatrix}.
\]

\begin{proof}[Proof of Theorem~$\ref{thm:tool}$]
Assume that the hypotheses of the theorem hold.
We fix $\theta\in\Tset^m$
 such that the integral \eqref{eqn:A2} is not zero.
We continue the fundamental matrix $\Phi^k(t;\theta)$ analytically
 along the loop $\gamma=\gamma_\theta$ 
 to obtain the monodromy matrix as
\begin{equation}
M_{\gamma}=
\begin{pmatrix}
\id_\ell & 0 & \hat{C}_1\\
0 & \id_m &  \hat{C}_2\\
0 & 0 & 1
\end{pmatrix},
\label{eqn:Mgamma}
\end{equation}
where
\begin{align*}
&
\hat{C}_1=\int_{\gamma_\theta}h^k(I^\ast,\omega(I^\ast)t+\theta;0)\d\tau,\\
&
\hat{C}_2=\int_{\gamma_\theta}\bigl(\D\omega(I^\ast)\Xi^k(\tau;\theta)
+g^k(I^\ast,\omega(I^\ast)\tau+\theta)\bigr)\d\tau.
\end{align*}
See Appendix~B for basic information on monodromy matrices.
In particular, we have $M_\gamma\in\G_\theta$.
Note that $\D\omega(I^\ast)\hat{C}_1\neq 0$ by assumption~(A2).

Let $\bar{\gamma}=\{T^\ast s\mid s\in[0,1]\}$,
 which is also a closed loop on the Riemann surface $\Gamma$
 (see Fig.~\ref{fig:2b}).
We continue $\Phi^k(t;\theta)$ analytically along the loop $\bar{\gamma}$
 to obtain the monodromy matrix as
\[
M_{\bar{\gamma}}=
\begin{pmatrix}
\id_\ell & 0 & \Xi^k(T^\ast;\theta)\\
\D\omega(I^\ast)T^\ast & \id_m & \Psi^k(T^\ast;\theta)\\
0 & 0 & 1
\end{pmatrix}.
\]
Let $\bar{C}_1=\Xi^k(T^\ast;\theta)$, $\bar{C}_2=\Psi^k(T^\ast;\theta)$
 and $\bar{C}_3=\D\omega(I^\ast)T^\ast$.
We see that $M_{\bar{\gamma}}=M(\bar{C}_1,\bar{C}_2,\bar{C}_3)\in\G_\theta$
 and $\bar{C}_3\hat{C}_1\neq 0$ by $\D\omega(I^\ast)\hat{C}_1\neq 0$.

\begin{lem}
\label{lem:2b}
Suppose that $M(C_1,C_2,C_3),M(C_1',C_2',C_3')\in\G_\theta$
 for some $C_j,C_j'$, $j=1,2,3$, with $C_3C_1'\neq C_3'C_1$.
Then the identity component $\G_\theta^0$ of $\G_\theta$ is not commutative.
\end{lem}

\begin{proof}
Assume that the hypothesis holds.
We easily see that $M(C_1,C_2,C_3)$ and $M(C_1',C_2',C_3')$ is not commutative since
\[
M(C_1,C_2,C_3)M(C_1',C_2',C_3)
=\begin{pmatrix}
\id_\ell & 0 & C_1+C_1'\\
C_3+C_3' & \id_m & C_3C_1'+C_2+C_2'\\
0 & 0 & 1
\end{pmatrix}
\]
while
\[
M(C_1',C_2',C_3)M(C_1',C_2',C_3)
=\begin{pmatrix}
\id_\ell & 0 & C_1+C_1'\\
C_3+C_3' & \id_m & C_3'C_1+C_2+C_2'\\
0 & 0 & 1
\end{pmatrix}.
\]
However, we compute
\begin{align*}
M(C_1,C_2,C_3)^2
=&\begin{pmatrix}
\id_\ell & 0 & C_1\\
C_3 & \id_m & C_2\\
0 & 0 & 1
\end{pmatrix}\begin{pmatrix}
\id_\ell & 0 & C_1\\
C_3 & \id_m & C_2\\
0 & 0 & 1
\end{pmatrix}\\
=&\begin{pmatrix}
\id_\ell & 0 & 2C_1\\
2C_3 & \id_m & C_3C_1+2C_2\\
0 & 0 & 1
\end{pmatrix}
\end{align*}
and easily show by induction that
\[
M(C_1,C_2,C_3)^k
=\begin{pmatrix}
\id_\ell & 0 & kC_1\\
kC_3 & \id_m & (k-1)C_3C_1+kC_2\\
0 & 0 & 1
\end{pmatrix}
\]
for any $k\in\Nset$.
Since $\G_\theta^0$ is a subgroup of finite index in $\G_\theta$ (see Appendix~A),
 we show that
\[
\G_\theta^0\supset\{M(cC_1,(c-1)C_3C_1+cC_2,cC_3)\mid c\in\Cset\}
\]
if $C_3C_1+C_2\neq 0$ and
\[
\G_\theta^0\supset\{M(cC_1,C_2,cC_3)\mid c\in\Cset\}
\]
if $C_3C_1+C_2=0$.
Thus, we show that $\G_\theta^0$ is not commutative in both cases.
\end{proof}

By Lemma~\ref{lem:2b},
 the identity component $\G_\theta^0$ is not commutative.
Applying Theorem~\ref{thm:MRS},
 we see that the system \eqref{eqn:aasys} is meromorphically nonintegrable
 near the resonant periodic orbit $(I^\ast,\omega(I^\ast)t+\theta)$
 in the meaning of Theorem~\ref{thm:tool}.
If this statement holds for $\theta$ on a dense set $\Delta\subset\Tset^m$,
 then so does it  on $\Tset^m$.
Thus, we complete the proof.
\end{proof}

\begin{rmk}\
\label{rmk:2a}
\begin{enumerate}
\setlength{\leftskip}{-1.6em}
\item[\rm(i)]
When the system \eqref{eqn:aasys} is Hamiltonian,
 it is not meromorphically Liouville-integrable
 such that the first integrals also depend meromorphically on $\epsilon$ near $\epsilon=0$,
 if the hypotheses of Theorem~$\ref{thm:tool}$ hold.
\item[\rm(ii)]
Assumption {\rm(A2)} in Theorem~$\ref{thm:tool}$ may be replaced with\\[-1ex]
\begin{enumerate}
\setlength{\leftskip}{-0.6em}
\item[\bf(A2')]
For some $k\in\Zset_{\ge 0}$ and $\theta\in\Tset^m$
\[
\D\omega(I^\ast)\Xi^k(t;\theta)\notin\Kset_\theta(t).
\]
\end{enumerate}
This is easily proven as follows.
Let $\Lset$ be the Picard-Vessiot extension of \eqref{eqn:rve}
 and let $\hat{\sigma}:\Lset\to\Lset$ be a $\Kset_\theta(t)$-automorphism, i.e.,
 $\hat{\sigma}\in\mathrm{Gal}(\Lset/\Kset_\theta(t))\subset\G_\theta$
 $($see Appendix~{\rm A)}.
Since $\Xi^k(t;\theta)\notin\Kset_\theta(t)$, we have
\[
\hat{\sigma}(\Xi^k(t;\theta))=\Xi^k(t;\theta)+\hat{C}_1
\]
as in \eqref{eqn:rmk2a},
 so that $\hat{\sigma}$ corresponds to the matrix
\[
\begin{pmatrix}
\id_\ell & 0 & \hat{C}_1\\
0 & \id_m & \hat{C}_2\\
0 & 0 & 1 
\end{pmatrix}=M(\hat{C}_1,\hat{C}_2,0).
\]
Since $\D\omega(I^\ast)\hat{C}_1\neq 0$
 for some $\hat{\sigma}\in\mathrm{Gal}(\Lset/\Kset_\theta(t))$,
 we only have to use the above matrix instead of \eqref{eqn:Mgamma}
 and apply the same arguments to obtain the desired result.
\end{enumerate}
\end{rmk}

\section{Planar Case}
We prove Theorem~\ref{thm:main} for the planar case \eqref{eqn:pp}.
We only consider a neighborhood of $(x,y)=(-\mu,0)$
  since we only have to replace $x$ and $\mu$ with $-x$ and $1-\mu$
  to obtain the result for a neighborhood of  $(1-\mu,0)$.
We introduce a small parameter $\epsilon$ such that $0<\epsilon\ll 1$.
Letting
\[
\epsilon^2\xi=x+\mu,\quad
\epsilon^2\eta=y,\quad
\epsilon^{-1}p_\xi=p_x,\quad
\epsilon^{-1}p_\eta=p_y+\mu
\]
and scaling the time variable $t\to\epsilon^3 t$, we rewrite \eqref{eqn:pp} as
\begin{align*}
&
\dot{\xi}=p_\xi+\epsilon^3\eta,\quad
\dot{p}_\xi=-\frac{(1-\mu)\xi}{(\xi^2+\eta^2)^{3/2}}
 +\epsilon^3 p_\eta-\epsilon^4\mu
 -\epsilon^4\frac{\mu(\epsilon^2\xi-1)}{((\epsilon^2\xi-1)^2+\epsilon^4\eta^2)^{3/2}},\\
&
\dot{\eta}=p_\eta-\epsilon^3\xi,\quad
\dot{p}_\eta=-\frac{(1-\mu)\eta}{(\xi^2+\eta^2)^{3/2}}
 -\epsilon^3 p_\xi-\epsilon^6\frac{\mu\eta}{((\epsilon^2\xi-1)^2+\epsilon^4\eta^2)^{3/2}},
\end{align*}
or up to the order of $\epsilon^6$,
\begin{equation}
\begin{split}
&
\dot{\xi}=p_\xi+\epsilon^3\eta,\quad
\dot{p}_\xi=-\frac{(1-\mu)\xi}{(\xi^2+\eta^2)^{3/2}}
 +\epsilon^3 p_\eta+2\epsilon^6\mu\xi,\\
&
\dot{\eta}=p_\eta-\epsilon^3\xi,\quad
\dot{p}_\eta=-\frac{(1-\mu)\eta}{(\xi^2+\eta^2)^{3/2}}-\epsilon^3 p_\xi
 -\epsilon^6\mu\eta,
\end{split}
\label{eqn:pped}
\end{equation}
where the $O(\epsilon^8)$ terms have been eliminated.
Equation~\eqref{eqn:pped} is a Hamiltonian system with the Hamiltonian
\begin{equation}
H=\tfrac{1}{2}(p_\xi^2+p_\eta^2)-\frac{1-\mu}{\sqrt{\xi^2+\eta^2}}
 +\epsilon^3(\eta p_\xi-\xi p_\eta)-\tfrac{1}{2}\epsilon^6\mu(2\xi^2-\eta^2).
\label{eqn:H1}
\end{equation}
Nonintegrability of a system
 which is similar to \eqref{eqn:pped} but does not contain a small parameter 
 was proven by using the Morales-Ramis theory \cite{M99,MR01} in \cite{MSS05}.
See also Remark~\ref{rmk:3a}(ii).

We next rewrite \eqref{eqn:H1} in the polar coordinates.
Let
\[
\xi=r\cos\phi,\quad
\eta=r\sin\phi.
\]
The momenta $(p_r,p_\phi)$ corresponding to $(r,\phi)$ satisfy
\[
p_\xi=p_r\cos\phi-\frac{p_\phi}{r}\sin\phi,\quad
p_\eta=p_r\sin\phi+\frac{p_\phi}{r}\cos\phi.
\]
See, e.g., Section~8.6.1 of \cite{MO17}.
The Hamiltonian becomes
\begin{equation*}
H=\tfrac{1}{2}\left(p_r^2+\frac{p_\phi^2}{r^2}\right)-\frac{1-\mu}{r}-\epsilon^3 p_\phi
 -\tfrac{1}{4}\epsilon^6\mu r^2(3\cos 2\phi+1).
\end{equation*}
Up to $O(1)$, the corresponding Hamiltonian system becomes
\begin{equation}
\dot{r}=p_r,\quad
\dot{p}_r=\frac{p_\phi^2}{r^3}-\frac{1-\mu}{r^2},\quad
\dot{\phi}=\frac{p_\phi}{r^2},\quad
\dot{p}_\phi=0,
\label{eqn:pp0p}
\end{equation}
which is easily solved since $p_\phi$ is a constant.
Let $u=1/r$.
From \eqref{eqn:pp0p} we have
\[
\frac{\d^2u}{\d\phi^2}+u=\frac{1-\mu}{p_\phi^2},
\]
from which we obtain the relation
\begin{equation}
r=\frac{p_\phi^2}{(1-\mu)(1+e\cos\phi)},
\label{eqn:ppr}
\end{equation}
where the position $\phi=0$ is appropriately chosen and $e$ is a constant.
We choose $e\in(0,1)$, so that Eq.~\eqref{eqn:ppr} represents an elliptic orbit with the eccentricity $e$.
Moreover, its period is given by
\begin{equation}
T=\frac{p_\phi^3}{(1-\mu)^2}\int_0^{2\pi}\frac{\d\phi}{(1+e\cos\phi)^2}
 =\frac{2\pi p_\phi^3}{(1-\mu)^2(1-e^2)^{3/2}}.
\label{eqn:T}
\end{equation}

Now we introduce  the Delaunay elements obtained from the generating function
\begin{equation}
W(r,\phi,I_1,I_2)=I_2\phi+\chi(r,I_1,I_2),
\label{eqn:W}
\end{equation}
where
\begin{align}
\chi(r,I_1,I_2)
=&\int_{r_-}^r\left(\frac{2(1-\mu)}{\rho}-\frac{(1-\mu)}{I_1^2}
 -\frac{I_2^2}{\rho^2}\right)^{1/2}\d\rho\notag\\
=&
 -2I_1^2\arcsin\sqrt{\frac{r_+-r}{r_+-r_-}}
 +\sqrt{(r_+-r)(r-r_-)}\notag\\
&
 +\frac{2I_1I_2}{\sqrt{1-\mu}}\arctan\sqrt{\frac{r_-(r_+-r)}{r_+(r-r_-)}}
\label{eqn:chi}
\end{align}
with
\[
r_\pm=I_1\left(I_1\pm\sqrt{I_1^2-\frac{I_2^2}{1-\mu}}\right)
\]
(see, e.g., Section~8.9.1 of \cite{MO17}).
We have
\begin{align*}
&
p_r=\frac{\partial W}{\partial r}
=\frac{\partial\chi}{\partial r}(r,I_1,I_2),\quad
p_\phi=\frac{\partial W}{\partial\phi}
=I_2,\\
&
\theta_1=\frac{\partial W}{\partial I_1}=\chi_1(r,I_1,I_2),\quad
\theta_2=\frac{\partial W}{\partial I_2}=\phi+\chi_2(r,I_1,I_2),
\end{align*}
where
\begin{align*}
\chi_1(r,I_1,I_2)=\frac{\partial\chi}{\partial I_1}(r,I_1,I_2),\quad
\chi_2(r,I_1,I_2)=\frac{\partial\chi}{\partial I_2}(r,I_1,I_2).
\end{align*}
Since the transformation from $(r,\phi,p_r,p_\phi)$ to $(\theta_1,\theta_2,I_1,I_2)$ is symplectic,
 the transformed system is also Hamiltonian and its Hamiltonian is given by
\begin{align*}
H=&-\frac{(1-\mu)}{2I_1^2}-\epsilon^3 I_2\notag\\
& -\tfrac{1}{4}\epsilon^6\mu R(\theta_1,I_1,I_2)^2
 (3\cos 2(\theta_2-\chi_2(R(\theta_1,I_1,I_2),I_1,I_2))+1),
\end{align*}
where $r=R(\theta_1,I_1,I_2)$ is the $r$-component of the symplectic transformation satisfying
\begin{equation}
\theta_1=\chi_1(R(\theta_1,I_1,I_2),I_1,I_2).
\label{eqn:R}
\end{equation}
Thus, we obtain the Hamiltonian system 
\begin{equation}
\begin{split}
\dot{I}_1=&
 \tfrac{1}{2}\epsilon^6\mu\frac{\partial R}{\partial\theta_1}(\theta_1,I_1,I_2)R(\theta_1,I_1,I_2)
 \biggl(3\cos 2(\theta_2-\chi_2(R(\theta_1,I_1,I_2),I_1,I_2))\\
&\quad
+1+3R(\theta_1,I_1,I_2)\frac{\partial\chi_2}{\partial r}(R(\theta_1,I_1,I_2),I_1,I_2)\\
&\quad
\times\sin 2(\theta_2-\chi_2(R(\theta_1,I_1,I_2),I_1,I_2))\biggr),\\
\dot{I}_2=&
 -\tfrac{3}{2}\epsilon^6\mu R(\theta_1,I_1,I_2)^2
 \sin 2(\theta_2-\chi_2(R(\theta_1,I_1,I_2),I_1,I_2)),\\
\dot{\theta}_1=&
\frac{1-\mu}{I_1^3}+O(\epsilon^6),\quad
\dot{\theta}_2=-\epsilon^3+O(\epsilon^6).
\end{split}
\label{eqn:ppe}
\end{equation}

Similarly to the treatment for \eqref{eqn:pp} stated just above Theorem~\ref{thm:main},
 the new variables $(v_1,v_2,v_3)\in\Cset\times(\Cset/2\pi\Zset)^2$ given by
\begin{align*}
V_1(v_1,r,I_1,I_2)
 :=&v_1^2+r^2-2I_1^2r+\frac{I_1^2I_2^2}{1-\mu}=0,\\
V_2(v_2,r,I_1,I_2)
 :=&I_1^2\left(I_1^2-\frac{I_2^2}{1-\mu}\right)(2\sin^2v_2-1)^2-(r-I_1^2)^2=0,\\
V_3(v_3,r,I_1,I_2)
 :=& I_1^2\left(r-\frac{I_2^2}{1-\mu}\right)^2(\tan^2v_3+1)^2\\
& -r^2\left(I_1^2-\frac{I_2^2}{1-\mu}\right)^2(\tan^2v_3-1)^2=0
\end{align*}
are introduced,
 so that the generating function \eqref{eqn:W} is regarded as an analytic one
 on the four-dimensional complex manifold
\begin{align*}
\bar{\S}_2=\{(r,\phi,I_1,I_2,v_1,v_2,v_3)
&\in\Cset\times(\Cset/2\pi\Zset)\times\Cset^3\times(\Cset/2\pi\Zset)^2\\
& \mid V_j(v_j,r,I_1,I_2)=0,\ j=1,2,3\}
\end{align*}
since Eq.~\eqref{eqn:chi} is represented by
\[
\chi(I_1,I_2,v_1,v_2,v_3)=v_1-2I_1^2v_2+\frac{2I_1I_2v_3}{\sqrt{1-\mu}}.
\]
Hence, we can regard \eqref{eqn:ppe}
 as a meromorphic two-degree-of-freedom Hamiltonian systems
 on the four-dimensional complex manifold
\begin{align*}
\hat{\S}_2=&\{(I_1,I_2,\theta_1,\theta_2,r,v_1,v_2,v_3)
\in\Cset^2\times(\Cset/2\pi\Zset)^2\times\Cset^2\times(\Cset/2\pi\Zset)^2\\
& \quad
 \mid\theta_1-\chi_1(r,I_1,I_2)=V_j(v_j,r,I_1,I_2)=0,\ j=1,2,3\},
\end{align*}
like \eqref{eqn:ppu} on $\S_2$ for \eqref{eqn:pp}.
Actually, we have
\begin{align*}
\frac{\partial V_j}{\partial v_j}\frac{\partial v_j}{\partial r}
 +\frac{\partial V_j}{\partial r}=0,\quad
\frac{\partial V_j}{\partial v_j}\frac{\partial v_j}{\partial I_l}
 +\frac{\partial V_j}{\partial I_l}=0,\quad
 j=1,2,3,\
 l=1,2,
\end{align*}
to express
\begin{align*}
\frac{\partial\chi}{\partial r}
 =\sum_{j=1}^3\frac{\partial\chi}{\partial v_j}\frac{\partial V_j}{\partial r},\quad
\chi_l=\frac{\partial\chi}{\partial I_l}
 +\sum_{j=1}^3\frac{\partial\chi}{\partial v_j}\frac{\partial V_j}{\partial I_l},\quad
 l=1,2
\end{align*}
as meromorophic functions of $(r,I_1,I_2,v_1,v_2,v_3)$ on $\bar{\S}_2$.
In particular, the Hamiltonian system has an additional first integral that is meromorphic
 in\linebreak
 $(I_1,I_2,\theta_1,\theta_2,v_1,v_2,v_3,\epsilon)$
 on $\hat{\S}_2\setminus\Sigma(\hat{\S}_2)$ near $\epsilon=0$
 if the system \eqref{eqn:pp} has an additional first integral
 that is meromorphic in $(x,y,p_x,p_y,u_1,u_2)$ on $\S_2\setminus\Sigma(\S_2)$
 near $(x,y)=(-\mu,0)$,
 as in Theorem~2 of \cite{C13},
 since the corresponding Hamiltonian system
 has the same expression as \eqref{eqn:ppe}
 on $\hat{\S}_2\setminus\Sigma(\hat{\S}_2)$,
 where $\Sigma(\hat{\S}_2)$ is the critical set of $\hat{\S}_2$
 on which the projection $\hat{\pi}_2:\hat{\S}_2\to\Cset^2\times(\Cset/2\pi\Zset)^2$
 given by
\[
\hat{\pi}_2(I_1,I_2,\theta_1,\theta_2,r,v_1,v_2,v_3)=(I_1,I_2,\theta_1,\theta_2)
\]
is singular.

We next estimate the $O(\epsilon^6)$-term in the first equation of \eqref{eqn:ppe}
 for the unperturbed solutions. 
When $\epsilon=0$, we see that $I_1,I_2,\theta_2$ are constants
 and can write $\theta_1=\omega_1 t+\theta_{10}$ for any solution to \eqref{eqn:ppe}, where
\begin{equation}
\omega_1=\frac{1-\mu}{I_1^3}
\label{eqn:omega1}
\end{equation}
and $\theta_{10}\in\Sset^1$ is a constant.
Since $r=R(\omega_1 t+\theta_{10},I_1,I_2)$ and
\[
\phi
 =-\chi_2(R(\omega_1 t+\theta_{10},I_1,I_2),I_1,I_2),
\]
respectively, become the $r$- and $\phi$-components of a solution to \eqref{eqn:pp0p},
 we have
\begin{equation}
\begin{split}
&
R(\omega_1 t+\theta_{10},I_1,I_2)
 =\frac{I_2^2}{(1-\mu)(1+e\cos(\phi(t)+\bar{\phi}(\theta_{10})))},\\
&
-\chi_2(R(\omega_1 t+\theta_{10},I_1,I_2),I_1,I_2)
 =\phi(t)+\bar{\phi}(\theta_{10})
\end{split}
\label{eqn:pprchi}
\end{equation}
by \eqref{eqn:ppr},
 where $\phi(t)$ is the $\phi$-component of a solution to \eqref{eqn:pp0p}
  and $\bar{\phi}(\theta_{10})$ is a constant depending on $\theta_{10}$. 
Differentiating both equations in \eqref{eqn:pprchi} with respect to $t$ yields
\begin{equation}
\begin{split}
&
\omega_1\frac{\partial R}{\partial\theta_1}(\omega_1t+\theta_{10},I_1,I_2)
=\frac{eI_2^2\sin(\phi(t)+\bar{\phi}(\theta_{10}))\dot{\phi}(t)}
 {(1-\mu)(1+e\cos(\phi(t)+\bar{\phi}(\theta_{10})))^2},\\
&
-\omega_1\frac{\partial\chi_2}{\partial r}
(R(\omega_1 t+\theta_{10},I_1,I_2),I_1,I_2)
\frac{\partial R}{\partial\theta_1}(\omega_1t+\theta_{10},I_1,I_2)
=\dot{\phi}(t).
\end{split}
\label{eqn:pprchi2}
\end{equation}
Using \eqref{eqn:pprchi} and \eqref{eqn:pprchi2},
 we can obtain the necessary expression of the $O(\epsilon^6)$-term.

We are ready to check the hypotheses of Theorem~\ref{thm:tool} for the system \eqref{eqn:ppe}.
Assumption~(A1) holds for any $I_1>0$.
Fix the values of $I_1,I_2$ at some $I_1^\ast,I_2^\ast>0$, and let $\omega^\ast=\omega_1/3$.
Since by the second equation of \eqref{eqn:pprchi} $\phi(t)$ is $2\pi/\omega_1$-periodic,
 we have
\[
\frac{2\pi I_2^{\ast 3}}{(1-\mu)^2(1-e^2)^{3/2}}=\frac{2\pi I_1^3}{1-\mu}
\]
by \eqref{eqn:T} and \eqref{eqn:omega1}, so that
\[
I_2^\ast=I_1^\ast(1-\mu)^{1/3}\sqrt{1-e^2}.
\]
From \eqref{eqn:omega1} we also have
\begin{align*}
\D\omega(I^\ast)
=&\begin{pmatrix}
-3(1-\mu)/I_1^{\ast 4} & 0\\
0 & 0
\end{pmatrix},
\end{align*}
where $I^\ast=(I_1^\ast,I_2^\ast)$.
We write the first component of \eqref{eqn:A2} with $k=5$ for $I=I^\ast$ as
\begin{align*}
\mathscr{I}_1^5(\theta)
=&-\frac{3\mu(1-\mu)}{2I_1^{\ast 4}}\int_{\gamma_\theta}\biggl(
\frac{\partial R}{\partial\theta_1}(\omega_1 t+\theta_1,I_1^\ast,I_2^\ast)
R(\omega_1 t+\theta_1,I_1^\ast,I_2^\ast)\notag\\
&
\times \biggl(3\cos 2(\theta_2
 -\chi_2(R(\omega_1 t+\theta_1,I_1^\ast,I_2^\ast),I_1^\ast,I_2^\ast))+1\notag\\
&
+3R(\omega_1 t+\theta_1,I_1^\ast,I_2^\ast)
 \frac{\partial\chi_2}{\partial r}
 (R(\omega_1 t+\theta_1,I_1^\ast,I_2^\ast),I_1^\ast,I_2^\ast)\notag\\
&
\times\sin 2(\theta_2
 -\chi_2(R(\omega_1 t+\theta_1,I_1^\ast,I_2^\ast),I_1^\ast,I_2^\ast))\biggr)\d t,
\end{align*}
where the closed loop $\gamma_\theta$ is specified below.
Using \eqref{eqn:pprchi} and \eqref{eqn:pprchi2},
 we compute 
\begin{align}
\mathscr{I}_1^5(\theta)=&\frac{3\mu I_2^{\ast 4}}{2(1-\mu)^2I_1^\ast}
\int_{\gamma_\theta}\dot{\phi}(t)\biggl(
\frac{3\sin 2(\phi(t)+\bar{\phi}(\theta_1)+\theta_2)}
 {(1+e\cos(\phi(t)+\bar{\phi}(\theta_1)))^2}\notag\\
&
-\frac{e\sin(\phi(t)+\bar{\phi}(\theta_1))
 (3\cos 2(\phi(t)+\bar{\phi}(\theta_1)+\theta_2)+1)}
 {(1+e\cos(\phi(t)+\bar{\phi}(\theta_1)))^3}\biggr)\d t.
\label{eqn:ppXi1}
\end{align}
By \eqref{eqn:pp0p} and \eqref{eqn:ppr} we have
\begin{equation}
\frac{\dot{\phi}(t)}{(1+e\cos\phi(t))^2}=\frac{(1-\mu)^2}{I_2^{\ast 3}}
=\frac{\omega_1}{(1-e^2)^{3/2}}.
\label{eqn:ppdphi}
\end{equation}
Using integration by substitution and the relation \eqref{eqn:ppdphi},
 we rewrite the above integral as
\begin{align}
\mathscr{I}_1^5(\theta)
=&\frac{9\mu I_2^\ast}{2I_1^\ast}\int_{\gamma_\theta}\biggl(\sin 2(\phi(t)+\theta_2)
 -\frac{e\sin\phi(t)(\cos 2(\phi(t)+\theta_2)+\tfrac{1}{3})}{1+e\cos\phi(t)}\biggr)\d t,
\label{eqn:ppXi2}
\end{align}
where the path of integration might change
 but the same notation $\gamma_\theta$ has still been used for it.

\begin{figure}[t]
\includegraphics[scale=0.7]{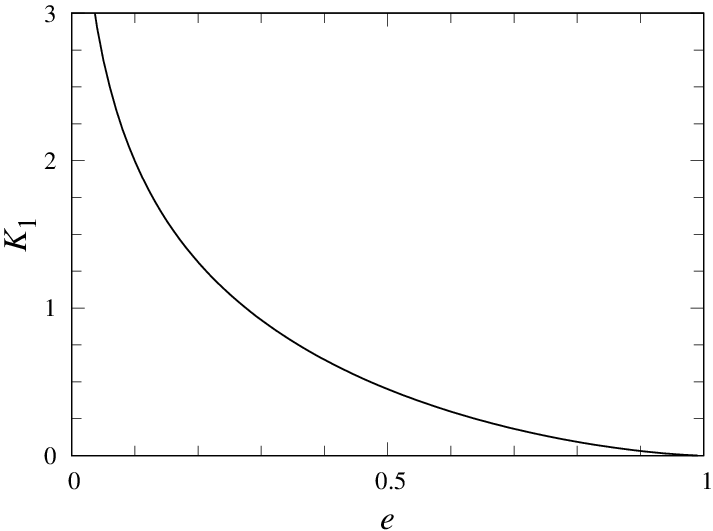}
\caption{Dependence of $K_1$ on $e$.
 \label{fig:3a}}
\end{figure}
Moreover, we integrate \eqref{eqn:ppdphi} to obtain
\begin{equation}
\omega_1 t=2\arctan\biggl(\frac{(1-e)\tan\tfrac{1}{2}\phi}{\sqrt{1-e^2}}\biggr)
 -\frac{e\sqrt{1-e^2}\sin\phi}{1+e\cos\phi}\quad
\mbox{for $\phi\in(-\pi,\pi)$},
\label{eqn:ppsol1}
\end{equation}
which is rewritten as
\begin{equation}
\omega_1 t=2\arccot\biggl(\frac{(1-e)\cot\tfrac{1}{2}(\phi+\pi)}{\sqrt{1-e^2}}\biggr)
 -\frac{e\sqrt{1-e^2}\sin\phi}{1+e\cos\phi}+\pi\quad
\mbox{for $\phi\in(0,2\pi)$},
\label{eqn:ppsol2}
\end{equation}
when $\phi(0)=0$ or $\lim_{t\to 0}\phi(t)=0$.
From \eqref{eqn:ppsol1} we see that as $\im\phi\to+\infty$, $\omega_1 t\to iK_1$, 
where
\[
K_1=2\arctanh\biggl(\frac{1-e}{\sqrt{1-e^2}}\biggr)-\sqrt{1-e^2}>0.
\]
See Fig.~\ref{fig:3a}.
So the integrand in  \eqref{eqn:ppXi2} is singular at $t=iK_1$.
Let $K_2=\arccosh(1/e)$.
Then $1+e\cos\phi=0$ at $\phi=\pi+iK_2$, and by \eqref{eqn:ppsol2}
\[
\frac{1}{\omega_1 t}=\frac{1}{\sqrt{1-e^2}}\Delta\phi+o(\Delta\phi)
\]
near $\phi=\pi+iK_2$, where $\Delta\phi=\phi-(\pi+iK_2)$.
Moreover, near $\phi=\pi+iK_2$,
\begin{align*}
&
\sin\phi=-\frac{i\sqrt{1-e^2}}{e}
 +O(\Delta\phi),\quad
\cos\phi=-\frac{1}{e}+\frac{i\sqrt{1-e^2}}{e}\Delta\phi
 +O(\Delta\phi^2),\\
&
\sin 2\phi=\frac{2i\sqrt{1-e^2}}{e^2}
 +O(\Delta\phi),\quad
\cos 2\phi=\frac{2-e^2}{e^2}
 +O(\Delta\phi).
\end{align*}

\begin{figure}[t]
\includegraphics[scale=1]{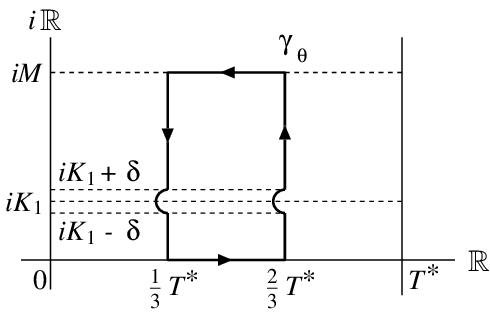}
\caption{Closed path $\gamma_\theta$.
 \label{fig:3b}}
\end{figure}

We take a closed path starting and ending at $t=\tfrac{1}{3}T^\ast$
 and passing through $t=\tfrac{2}{3}T^\ast$, $\tfrac{2}{3}T^\ast+i(K_1\mp\delta)$,
 $\tfrac{2}{3}T^\ast+iM$, $\tfrac{1}{3}T^\ast+iM$ and $\tfrac{1}{3}T^\ast+i(K_1\pm\delta)$
 as $\gamma_\theta$ in $\Cset/T^\ast\Zset$,
 where $\delta$ and $M$ are, respectively, sufficiently small and large positive constants.
See Fig.~\ref{fig:3b}.
Here $\gamma_\theta$ passes along the left circular arc centered
 at $\tfrac{2}{3}T^\ast+iK_1$ (resp. at $\tfrac{1}{3}T^\ast+iK_1$) with radius $\delta$
 between $\tfrac{2}{3}T^\ast+i(K_1-\delta)$ and $\tfrac{2}{3}T^\ast+i(K_1+\delta)$
 (resp. between $\tfrac{1}{3}T^\ast+i(K_1+\delta)$ and $\tfrac{1}{3}T^\ast+i(K_1-\delta)$).
We compute
\begin{align*}
&
\int_{2T^\ast\!/3+iM}^{T^\ast\!/3+iM}
 \frac{\sin\phi(t)(\cos 2(\phi(t)+\theta_2)+\tfrac{1}{3})}{1+e\cos\phi(t)}\d t\\
&
=-\int_{2T^\ast\!/3}^{T^\ast\!/3}\biggl(\frac{2-e^2}{e^3\sqrt{1-e^2}}\cos 2\theta_2
 -\frac{2i}{e^3}\sin 2\theta_2+\frac{1}{3e\sqrt{1-e^2}}\biggr)i\omega_1 M\,\d t+O(1)\\
&
=\frac{2\pi}{e^3}\biggl(\frac{2-e^2}{\sqrt{1-e^2}}i\cos2\theta_2+2\sin2\theta_2
 +\frac{ie^2}{3\sqrt{1-e^2}}\biggr)M+O(1),
\end{align*}
while
\[
\int_{2T^\ast\!/3+iM}^{T^\ast\!/3+iM}\sin 2(\phi(t)+\theta_2)\d t=O(1).
\]
Moreover, the integral on $[\tfrac{1}{3}T^\ast,\tfrac{2}{3}T^\ast]$ in \eqref{eqn:ppXi2} is $O(1)$,
 and the integrals from $\tfrac{2}{3}T^\ast$ to $\tfrac{2}{3}T^\ast+iM$
 and from $\tfrac{1}{3}T^\ast+iM$ to $\tfrac{1}{3}T^\ast$ cancel
 since the integrand is $\tfrac{1}{3}T^\ast$-periodic.. 
Thus, we see that the integral \eqref{eqn:ppXi2} is not zero
 for $M>0$ sufficiently large,
 so that assumption~(A2) holds.

Finally, we apply Theorem~\ref{thm:tool}
 to show that the meromorphic Hamiltonian system corresponding to \eqref{eqn:ppe}
 is not meromorphically integrable
 such that the first integral depends meromorphically on $\epsilon$ near $\epsilon=0$
 even if any higher-order terms are included.
Thus, we obtain the conclusion of Theorem~\ref{thm:main} for the planar case.
\hfill\qed

\begin{rmk}\
\label{rmk:3a}
\begin{enumerate}
\setlength{\leftskip}{-1.6em}
\item[\rm(i)]
The reader may think that a small circle centered
 at $t=\tfrac{1}{3}T^\ast+iK_1$ or $\tfrac{2}{3}T^\ast+iK_1$ can be taken as $\gamma_\theta$
 in the proof,
 since the integrand in \eqref{eqn:ppXi2} is singular there.
However, the integral \eqref{eqn:ppXi2} for the path is estimated to be zero
 $($cf. Section~$3$ of {\rm\cite{Y21b}}$)$.
\item[\rm(ii)]
The different change of coordinates
\[
\epsilon\xi=x+\mu,\quad
\epsilon\eta=y,\quad
p_\xi=p_x,\quad
p_\eta=p_y+\mu
\]
in \eqref{eqn:pp} yields
\begin{equation}
\begin{split}
&
\dot{\xi}=p_\xi+\epsilon\eta,\quad
\dot{p}_\xi=\epsilon p_\eta-\epsilon^{-1}\frac{(1-\mu)\xi}{(\xi^2+\eta^2)^{3/2}}
 +2\epsilon^2\mu\xi,\\
&
\dot{\eta}=p_\eta-\epsilon\xi,\quad
\dot{p}_\eta=-\epsilon p_\xi-\epsilon^{-1}\frac{(1-\mu)\eta}{(\xi^2+\eta^2)^{3/2}}
 -\epsilon^2\mu\eta
\end{split}
\label{eqn:rmk1}
\end{equation}
up to $O(\epsilon^2)$ after the time scaling $t\to t/\epsilon$.
As in {\rm\cite{MSS05}} we use the Levi-Civita regularization
\[
\begin{pmatrix}
\xi\\
\eta
\end{pmatrix}
=
\begin{pmatrix}
q_1 & -q_2\\
q_2 & q_1
\end{pmatrix}
\begin{pmatrix}
q_1\\
q_2
\end{pmatrix},\quad
\begin{pmatrix}
p_\xi\\
p_\eta
\end{pmatrix}
=\frac{2}{q_1^2+q_2^2}
\begin{pmatrix}
q_1 & -q_2\\
q_2 & q_1
\end{pmatrix}
\begin{pmatrix}
p_1\\
p_2
\end{pmatrix},
\]
$($see, e.g.,{\rm \cite{SS00}} or Section~$8.8.1$ of {\rm\cite{MO17}}$)$ to obtain
\begin{align*}
H+C_0
=& \frac{4}{q_1^2+q_2^2}\bigl(\tfrac{1}{4}C_0(q_1^2+q_2^2)+\tfrac{1}{2}(p_1^2+p_2^2)
 +\tfrac{1}{2}\epsilon(q_1^2+q_2^2)(q_2p_1-q_1p_2)\\
&
 -\tfrac{1}{2}\epsilon^2\mu(q_1^2+q_2^2)(q_1^4q_1^2q_2^2-4+q_2^4))
 -\tfrac{1}{4}\epsilon^{-1}(1-\mu)\bigr),
\end{align*}
which yields
\begin{align*}
H+C_0=& \frac{4}{q_1^2+q_2^2}\bigl(\tfrac{1}{4}C_0(q_1^2+q_2^2)+\tfrac{1}{2}(p_1^2+p_2^2)
 +\tfrac{1}{2}(q_1^2+q_2^2)(q_2p_1-q_1p_2)\\
&
 -\tfrac{1}{2}\mu(q_1^2+q_2^2)(q_1^4q_1^2q_2^2-4+q_2^4))
 -\tfrac{1}{4}(1-\mu)\bigr)
\end{align*}
after the scaling $(q,p)\to(q,p)/\epsilon^{3/2}$.
Using the approach of {\rm\cite{MSS05}},
 we can show that the Hamiltonian system with the Hamiltonian
\begin{align*}
\tilde{H}=&\tfrac{1}{4}C_0(q_1^2+q_2^2)+\tfrac{1}{2}(p_1^2+p_2^2)
 +\tfrac{1}{2}(q_1^2+q_2^2)(q_2p_1-q_1p_2)\\
&
 -\tfrac{1}{2}\mu(q_1^2+q_2^2)(q_1^4q_1^2q_2^2-4+q_2^4)
\end{align*}
is meromorphically nonintegrable.
This implies that the Hamiltonian \eqref{eqn:rmk1} is also meromorphically nonintegrable
 for $\epsilon>0$ fixed.
\end{enumerate}
\end{rmk}

\section{Spatial Case}

We prove Theorem~\ref{thm:main} for the spatial case \eqref{eqn:sp}.
As in the planar case,
 we only consider a neighborhood of $(x,y,z)=(-\mu,0,0)$
  and introduce a small parameter $\epsilon$ such that $0<\epsilon\ll 1$.
Letting
\begin{align*}
&
\epsilon^2\xi=x+\mu,\quad
\epsilon^2\eta=y,\quad
\epsilon^2\zeta=z,\\
&
\epsilon^{-1}p_\xi=p_x,\quad
\epsilon^{-1}p_\eta=p_y+\mu,\quad
\epsilon^{-1}p_\zeta=p_z
\end{align*}
and scaling the time variable $t\to\epsilon^3 t$, we rewrite \eqref{eqn:sp} as
\begin{align*}
&
\dot{\xi}=p_\xi+\epsilon^3\eta,\quad
\dot{\eta}=p_\eta-\epsilon^3\xi,\quad
\dot{\zeta}=p_\zeta,\\
&
\dot{p}_\xi=-\frac{(1-\mu)\xi}{(\xi^2+\eta^2+\zeta^2)^{3/2}}
 +\epsilon^3p_\eta-\epsilon^4\mu-\epsilon^4\frac{\mu(\epsilon^2\xi-1)}
 {((\epsilon^2\xi-1)^2+\epsilon^4(\eta^2+\zeta^2))^{3/2}},\\
&
\dot{p}_\eta=-\frac{(1-\mu)\eta}{(\xi^2+\eta^2+\zeta^2)^{3/2}}
 -\epsilon^3 p_\xi-\epsilon^6\frac{\mu\eta}
 {((\epsilon^2\xi-1)^2+\epsilon^4(\eta^2+\zeta^2))^{3/2}},\\
&
\dot{p}_\zeta=-\frac{(1-\mu)\zeta}{(\xi^2+\eta^2+\zeta^2)^{3/2}}
 -\epsilon^6\frac{\mu\zeta}
 {((\epsilon^2\xi-1)^2+\epsilon^4(\eta^2+\zeta^2)))^{3/2}},
\end{align*}
or up to the order of $\epsilon^6$,
\begin{equation}
\begin{split}
&
\dot{\xi}=p_\xi+\epsilon^3\eta,\quad
\dot{p}_\xi=-\frac{(1-\mu)\xi}{(\xi^2+\eta^2+\zeta^2)^{3/2}}
 +\epsilon^3p_\eta+2\epsilon^6\mu\xi,\\
&
\dot{\eta}=p_\eta-\epsilon^3\xi,\quad
\dot{p}_\eta=-\frac{(1-\mu)\eta}{(\xi^2+\eta^2+\zeta^2)^{3/2}}
 -\epsilon^3p_\xi-\epsilon^6\mu\eta,\\
&
\dot{\zeta}=p_\zeta,\quad
\dot{p}_\zeta=-\frac{(1-\mu)\zeta}{(\xi^2+\eta^2+\zeta^2)^{3/2}}-\epsilon^6\mu\zeta,
\end{split}
\label{eqn:sped}
\end{equation}
like \eqref{eqn:pped}, where the $O(\epsilon^8)$ terms have been eliminated.
Equation~\eqref{eqn:sped} is a Hamiltonian system with the Hamiltonian
\begin{align}
H=&\tfrac{1}{2}(p_\xi^2+p_\eta^2+p_\zeta^2)-\frac{1-\mu}{\sqrt{\xi^2+\eta^2+\zeta^2}}\notag\\
& +\epsilon^3(\eta p_\xi-\xi p_\eta)-\tfrac{1}{2}\epsilon^6(2\xi^2-\eta^2-\zeta^2).
 \label{eqn:H3}
\end{align}

\begin{figure}[t]
\includegraphics[scale=0.8]{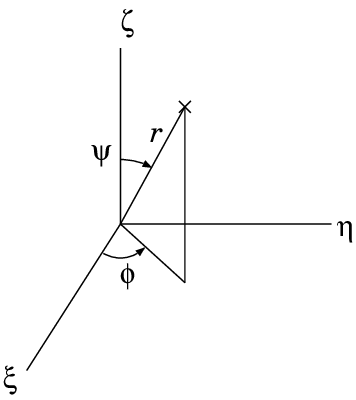}
\caption{Spherical coordinates.
 \label{fig:4a}}
\end{figure}

We next rewrite \eqref{eqn:H3} in the spherical coordinates.
See Fig.~\ref{fig:4a}.
Let
\[
\xi=r\sin\psi\cos\phi,\quad
\eta=r\sin\psi\sin\phi,\quad
\zeta=r\cos\psi.
\]
The momenta $(p_r,p_\phi,p_\psi)$ corresponding to $(r,\phi,\psi)$ satisfy
\begin{align*}
&
p_\xi=p_r\cos\phi\sin\psi-\frac{p_\phi\sin\phi}{r\sin\psi}+\frac{p_\psi}{r}\cos\phi\cos\psi,\\
&
p_\eta=p_r\sin\phi\sin\psi+\frac{p_\phi\cos\phi}{r\sin\psi}+\frac{p_\psi}{r}\sin\phi\cos\psi,\\
&
p_\zeta=p_r\cos\psi-\frac{p_\psi}{r}\sin\psi
\end{align*}
(see, e.g., Section~8.7 of \cite{MO17}).
The Hamiltonian becomes
\begin{align*}
H=&\tfrac{1}{2}\left(p_r^2+\frac{p_\psi^2}{r^2}+\frac{p_\phi^2}{r^2\sin^2\psi}\right)
 -\frac{1-\mu}{r}\notag\\
& -\epsilon^3 p_\phi
 -\epsilon^6\mu r^2\left(\tfrac{1}{4}\sin^2\psi(3\cos 2\phi+1)-\tfrac{1}{2}\cos^2\psi \right).
\end{align*}
Up to $O(1)$, the corresponding Hamiltonian system becomes
\begin{equation}
\begin{split}
&
\dot{r}=p_r,\quad
\dot{p}_r=\frac{p_\psi^2}{r^3}+\frac{p_\phi^2}{r^3\sin^2\psi}-\frac{1-\mu}{r^2},\\
&
\dot{\phi}=\frac{p_\phi}{r^2\sin^2\psi},\quad
\dot{p}_\phi=0,\quad
\dot{\psi}=\frac{p_\psi}{r^2},\quad
\dot{p}_\psi=\frac{p_\phi^2\cos\psi}{r^2\sin^3\psi}.
\label{eqn:sp0p}
\end{split}
\end{equation}
We have the relation \eqref{eqn:ppr} for periodic orbits on the $(\xi,\eta)$-plane
 since Eq.~\eqref{eqn:sp0p} reduces to \eqref{eqn:pp0p}
 when $\psi=\tfrac{1}{2}\pi$ and $p_\psi=0$.

As in the planar case,
 we introduce  the Delaunay elements obtained from the generating function
\begin{equation}
\hat{W}(r,\phi,\psi,I_1,I_2.I_3)=I_3\phi+\chi(r,I_1,I_2)+\hat{\chi}(\psi,I_2,I_3),
\label{eqn:hW}
\end{equation}
where
\begin{align}
\hat{\chi}(\psi,I_2,I_3)
=&\int_{\psi_0}^\psi\left(I_2^2-\frac{I_3^2}{\sin^2s}\right)^{1/2}\d s\notag\\
=& I_2\arctan\frac{\sqrt{I_2^2\sin^2\psi-I_3^2}}{I_2\cos\psi}
 -I_3\arctan\frac{\sqrt{I_2^2\sin^2\psi-I_3^2}}{I_3\cos\psi}
\label{eqn:hchi}
\end{align}
with $\psi_0=\arcsin(I_3/I_2)$.
See, e.g., Section~8.9.3 of \cite{MO17},
 although a slightly modified generating function is used here.
We have
\begin{equation}
\begin{split}
&
p_r=\frac{\partial \hat{W}}{\partial r}
=\frac{\partial\chi}{\partial r}(r,I_1,I_2),\quad
p_\phi=\frac{\partial \hat{W}}{\partial\phi}
=I_3,\\
&
p_\psi=\frac{\partial\hat{W}}{\partial\psi}
=\frac{\partial\hat{\chi}}{\partial\psi}(\psi,I_2,I_3),\quad
\theta_1=\frac{\partial \hat{W}}{\partial I_1}
=\chi_1(r,I_1,I_2),\\
&
\theta_2=\frac{\partial\hat{W}}{\partial I_2}
=\chi_2(r,I_1,I_2)+\hat{\chi}_2(\psi,I_2,I_3),\quad
\theta_3=\frac{\partial \hat{W}}{\partial I_3}
=\phi+\hat{\chi}_3(\psi,I_2,I_3),
\end{split}
\label{eqn:spd}
\end{equation}
where
\begin{align*}
\hat{\chi}_2(\psi,I_2,I_3)
=\frac{\partial\hat{\chi}}{\partial I_2}(\psi,I_2,I_3),\quad
\hat{\chi}_3(\psi,I_2,I_3)
=\frac{\partial\hat{\chi}}{\partial I_3}(\psi,I_2,I_3).
\end{align*}
Since the transformation from $(r,\phi,\psi,p_r,p_\phi,p_\psi)$
 to $(\theta_1,\theta_2,\theta_3,I_1,I_2,I_3)$ is symplectic,
 the transformed system is also Hamiltonian and its Hamiltonian is given by
\begin{align*}
H=&-\frac{(1-\mu)}{2I_1^2}-\epsilon^3 I_3
 -\epsilon^6\mu R(\theta_1,I_1,I_2)^2(\tfrac{1}{4}\sin^2\Psi(\theta_1,\theta_2,I_1,I_2,I_3)\\
&\quad
\times (3\cos 2(\theta_3-\hat{\chi}_3(\Psi(\theta_1,\theta_2,I_1,I_2,I_3),I_1,I_2))+1)\\
&\quad 
-\tfrac{1}{2}\cos^2\Psi(\theta_1,\theta_2,I_1,I_2,I_3)),
\end{align*}
where $r=R(\theta_1,I_1,I_2)$ and $\psi=\Psi(\theta_1,\theta_2,I_1,I_2,I_3)$
 are the $r$- and $\psi$-components of the symplectic transformation
 satisfying \eqref{eqn:R} and
\[
\hat{\chi}_2(\Psi(\theta_1,\theta_2,I_1,I_2,I_3),I_2,I_3)+\chi_2(R(\theta_1,I_1,I_2),I_1,I_2)
 =\theta_2,
\]
respectively.
Thus, we obtain the Hamiltonian system as
\begin{equation}
\begin{split}
\dot{I}_1=& \tfrac{1}{2}\epsilon^6\mu\hat{h}_1(I,\theta),\quad
\dot{I}_2=O(\epsilon^6),\quad
\dot{I}_3=O(\epsilon^6),\\
\dot{\theta}_1=& \frac{1-\mu}{I_1^3}+O(\epsilon^6),\quad
\dot{\theta}_2=O(\epsilon^6),\quad
\dot{\theta}_3=-\epsilon^3+O(\epsilon^6),
\end{split}
\label{eqn:spe}
\end{equation}
where
\begin{align*}
\hat{h}_1(I,\theta)
=& \frac{\partial R}{\partial\theta_1}(\theta_1,I_1,I_2)R(\theta_1,I_1,I_2)
 (\sin^2\Psi(\theta_1,\theta_2,I_1,I_2,I_3)\\
&\quad
\times(3\cos 2(\theta_3-\hat{\chi}_3(\Psi(\theta_1,\theta_2,I_1,I_2,I_3),I_2,I_3))+1)\\
&\quad
-2\cos^2\Psi(\theta_1,\theta_2,I_1,I_2,I_3))\\
&
+R(\theta_1,I_1,I_2)^2\frac{\partial\Psi}{\partial\theta_1}(\theta_1,\theta_2,I_1,I_2,I_3)\\
&\quad
\times 3\sin\Psi(\theta_1,\theta_2,I_1,I_2,I_3)\cos\Psi(\theta_1,\theta_2,I_1,I_2,I_3)\\
&\quad
\times(\cos 2(\theta_3-\hat{\chi}_3(\Psi(\theta_1,\theta_2,I_1,I_2,I_3),I_2,I_3))+1)\\
&
+3R(\theta_1,I_1,I_2)^2\sin^2\Psi(\theta_1,\theta_2,I_1,I_2,I_3)\\
&\quad
\times\frac{\partial\Psi}{\partial\theta_1}(\theta_1,\theta_2,I_1,I_2,I_3)
 \frac{\partial\hat{\chi}_3}{\partial\psi}(\Psi(\theta_1,\theta_2,I_1,I_2,I_3),I_2,I_3)\\
& \quad
\times\sin2(\theta_3-\hat{\chi}_3(\Psi(\theta_1,\theta_2,I_1,I_2,I_3),I_2,I_3)).
\end{align*}

As in the planar case, the new variables $w_1,w_2\in(\Cset/2\pi)$ given by
\begin{align*}
&
W_1(w_1,\psi,I_2,I_3):=I_2^2\cos^2\psi\tan^2w_1-I_2^2\sin^2\psi+I_3^2=0,\\
&
W_2(w_2,\psi,I_2,I_3):=I_3^2\cos^2\psi\tan^2w_2-I_2^2\sin^2\psi+I_3^2=0
\end{align*}
are introduced,
 so that the generating function \eqref{eqn:hW} is regarded as an analytic one
 on the six--dimensional complex manifold
\begin{align*}
\bar{\S}_3=&\{(r,\phi,\psi,I_1,I_2,I_3,v_1,v_2,v_3,w_1,w_2)
\in\Cset\times(\Cset/2\pi\Zset)^2\times\Cset^4\times(\Cset/2\pi\Zset)^4\\
&\quad
\mid V_j(v_j,r,I_1,I_2)=W_l(w_l,\psi,I_2,I_3)=0,\ j=1,2,3,\ l=1,2\}
\end{align*}
since Eq.~\eqref{eqn:hchi} is represented by
\[
\hat{\chi}(I_2,I_3;v_1,v_2)=I_2w_1-I_3w_2.
\]
Moreover, we can regard \eqref{eqn:spe}
 as a meromorphic three-degree-of-freedom Hamiltonian systems
 on the six-dimensional complex manifold
\begin{align*}
\hat{\S}_3=&\{(I_1,I_2,I_3,\theta_1,\theta_2,\theta_3,r,v_1,v_2,v_3,w_1,w_2)
 \in\Cset^3\times(\Cset/2\pi\Zset)^3\times\Cset^2\times(\Cset/2\pi\Zset)^4\\
&\quad
 \mid \theta_1-\chi_1(r,I_1,I_2)=V_j(v_j,r,I_1,I_2)=W_l(w_l,\psi,I_2,I_3)=0,\\
&\qquad
  j=1,2,3,\ l=1,2\},
\end{align*}
like \eqref{eqn:spu} on $\S_3$ for \eqref{eqn:sp}.
Actually, we have
\begin{align*}
\frac{\partial W_j}{\partial w_j}\frac{\partial w_j}{\partial\psi}
 +\frac{\partial W_j}{\partial\psi}=0,\quad
\frac{\partial W_j}{\partial w_j}\frac{\partial w_j}{\partial I_l}
 +\frac{\partial W_j}{\partial I_l}=0,\quad
 j=1,2,\ l=2,3
\end{align*}
to express
\begin{align*}
\frac{\partial\hat{\chi}}{\partial\psi}
 =\sum_{j=1}^2\frac{\partial\hat{\chi}}{\partial w_j}\frac{\partial W_j}{\partial\psi},\quad
\hat{\chi}_l=\frac{\partial\hat{\chi}}{\partial I_l}
 +\sum_{j=1}^2\frac{\partial\hat\chi}{\partial w_j}\frac{\partial W_j}{\partial I_l},\quad
 l=2,3
\end{align*}
as meromorophic functions of $(\psi,I_2,I_3,w_1,w_2)$ on $\bar{\S}_3$.
In particular, the Hamiltonian system has two additional meromorphic integrals
 that are meromorphic
 in $(I_1,I_2,I_3,\theta_1,\theta_2,\theta_3,r,v_1,v_2,v_3,w_1,w_ 2,\epsilon)$
 on $\hat{\S}_3\setminus\Sigma(\hat{\S}_3)$ near $\epsilon=0$,
 if the system \eqref{eqn:sp} has two additional meromorphic integrals
 that are meromorphic in\linebreak
 $(x,y,z,p_x,p_y,p_z,u_1,u_2)$
 on $\S_3\setminus\Sigma(\S_3)$ near $(x,y,z)=(-\mu,0,0)$,
 as in the planar case.
Here $\Sigma(\hat{\S}_3)$ is the critical set of $\hat{\S}_3$
 on which the projection $\hat{\pi}_3:\hat{\S}_3\to\Cset^3\times(\Cset/2\pi\Zset)^3$
 given by
\[
\hat{\pi}_3(I_1,I_2,I_3,\theta_1,\theta_2,\theta_3,r,v_1,v_2,v_3,w_1,w_2)
 =(I_1,I_2,I_3,\theta_1,\theta_2,\theta_3)
\]
is singular.

We next estimate the function $\hat{h}_1(I,\theta)$
 for solutions to \eqref{eqn:spe} with $\epsilon=0$ on the plane of $\psi=\tfrac{1}{2}\pi$.
When $\epsilon=0$, we see that $I_1,I_2,I_3,\theta_2,\theta_3$ are constants
 and can write $\theta_1=\omega_1 t+\theta_{10}$
 for any solution to \eqref{eqn:spe} with \eqref{eqn:omega1},
 where $\theta_{10}\in\Sset^1$ is a constant.
Note that if $\psi=\tfrac{1}{2}\pi$ and $p_\psi=0$, then $I_2=I_3$ by \eqref{eqn:spd}.
Since $r=R(\omega_1 t+\theta_{10},I_1,I_2)$ and
\[
\phi=-\hat{\chi}_3(\Psi(\omega_1 t+\theta_{10},\theta_2,I_1,I_3,I_3),I_3,I_3),\quad
\Psi(\omega_1 t+\theta_{10},\theta_2,I_1,I_3,I_3)=\tfrac{1}{2}\pi,
\]
respectively,
 become the $r$- and $\phi$-components of a solution to \eqref{eqn:sp0p}
 with $\psi=\tfrac{1}{2}\pi$ and $p_\psi=0$,
 we have the first equation of \eqref{eqn:pprchi} with
\begin{equation}
-\hat{\chi}_3(\Psi(\omega_1 t+\theta_{10},
 \theta_2,I_1,I_3,I_3),I_3,I_3)
 =\phi(t)+\bar{\phi}(\theta_{10}),
\label{eqn:sprchi}
\end{equation}
where $\phi(t)$ is the $\phi$-component of a solution to \eqref{eqn:pp0p}
 and $\bar{\phi}(\theta_1)$ is a constant depending only on $\theta_1$ as in the planar case.
Differentiating \eqref{eqn:sprchi} with respect to $t$ yields
\begin{align}
&
-\omega_1\frac{\partial \hat{\chi}_3}{\partial\psi}
 (\Psi(\omega_1 t+\theta_{10},
 \theta_2,I_1,I_3,I_3),I_3,I_3)\notag\\
&\qquad
\times\frac{\partial\Psi}{\partial\theta_1}(\omega_1 t+\theta_{10},
 \theta_2,I_1,I_3,I_3)=\dot{\phi}(t).
\label{eqn:sprchi2}
\end{align}
Using \eqref{eqn:pprchi}, \eqref{eqn:pprchi2}, \eqref{eqn:sprchi} and \eqref{eqn:sprchi2},
 we can obtain the necessary expression of $\hat{h}_1(I,\theta)$.

We are ready to check the hypotheses of Theorem~\ref{thm:tool} for the system {eqn:spe}.
Assumption~(A1) holds for any $I_1>0$.
Fix the value of $I_1$ at some $I_1^\ast>0$, and let $\omega^\ast=\omega_1/3$.
By the first equation of \eqref{eqn:sprchi} $\phi(t)$ is $2\pi/\omega_1$-periodic,
 so that by \eqref{eqn:T} and \eqref{eqn:omega1}
\[
I_2=I_3=I_1^\ast(1-\mu)^{1/3}\sqrt{1-e^2}\ (=I_2^\ast).
\]
From \eqref{eqn:omega1} we have
\[
\D\omega(I^\ast)=
\begin{pmatrix}
-3(1-\mu)/I_1^{\ast 4} & 0 & 0\\
0 & 0 & 0\\
0 & 0 & 0
\end{pmatrix},
\]
where $I^\ast=(I_1^\ast,I_2^\ast,I_2^\ast)$.
Using the first equations of \eqref{eqn:pprchi} and \eqref{eqn:pprchi2},
 \eqref{eqn:sprchi} and \eqref{eqn:sprchi2},
 we compute the first component of \eqref{eqn:A2} with $k=5$ for $I=I^\ast$ as
\begin{align*}
\mathscr{I}_1^5(\theta)
=&-\frac{3(1-\mu)}{I_1^{\ast 4}}\int_{\gamma_\theta}
 h_1(I^\ast,\omega^\ast t+\theta_1,\theta_2,\theta_3)\d\omega^\ast t\\
=&\frac{3\mu I_2^{\ast 4}}{2(1-\mu)^2I_1^\ast}
 \int_{\gamma_\theta}\dot{\phi}(t)\biggl(
\frac{3\sin 2(\phi(t)+\bar{\phi}(\theta_1)+\theta_3)}
 {(1+e\cos(\phi(t)+\bar{\phi}(\theta_1)))^2}\\
&
+\frac{e\sin(\phi(t)+\bar{\phi}(\theta_1))
 (3\cos 2(\phi(t)+\bar{\phi}(\theta_1)+\theta_3)+1)}
 {(1+e\cos(\phi(t)+\bar{\phi}(\theta_1)))^3}\biggr)\d t,
\end{align*}
which has the same expression as \eqref{eqn:ppXi1} with $\theta_2=\theta_3$.
Repeating the arguments given in Section~3,
 we can show that assumption~(A2) holds as in the planar case.
Finally, we apply Theorem~\ref{thm:tool}
 to show that the meromorphic Hamiltonian system corresponding to \eqref{eqn:spe}
 is not meromorphically integrable
 such that the first integrals depend meromorphically on $\epsilon$ near $\epsilon=0$.
Thus, we complete the proof of Theorem~\ref{thm:main} for the spatial case.
\hfill\qed
 
\section*{Acknowledgements}
The author thanks Mitsuru Shibayama, Shoya Motonaga and Taiga Kurokawa
 for helpful discussions, and David Bl\'azquez-Sanz for his useful comments.
This work was partially supported by the JSPS KAKENHI Grant Number JP17H02859.



\appendix

\renewcommand{\theequation}{A.\arabic{equation}}
\setcounter{equation}{0}

\section{Differential Galois Theory}
In this appendix,
 we give necessary information on differential Galois theory for linear differential equations,
 which is often referred to as the Picard-Vessiot theory.
See the textbooks \cite{CH11,PS03} for more details on the theory.

Consider a linear system of differential equations
\begin{equation}\label{LinearSystem}
y'=Ay,\quad A\in\mathrm{gl}(n,\Kset),
\end{equation}
where $\Kset$ is a differential field and
 $\mathrm{gl}(n,\Kset)$ denotes the ring of $n\times n$ matrices
 with entries in $\Kset$.
Here a \emph{differential field} is a field
 endowed with a derivation $\partial$,
 which is an additive endomorphism
 satisfying the Leibniz rule.
The set $\mathrm{C}_{\Kset}$ of elements of $\Kset$ for which $\partial$ vanishes
 is a subfield of $\Kset$
and called the \emph{field of constants of $\Kset$}.
In our application of the theory in this paper,
 the differential field $\Kset$ is
 the field of meromorphic functions on a Riemann surface,
 so that the field of constants is $\Cset$.

A \emph{differential field extension} $\Lset\supset \Kset$
 is a field extension such that $\Lset$ is also a differential field
 and the derivations on $\Lset$ and $\Kset$ coincide on $\Kset$.
A differential field extension $\Lset\supset \Kset$
 satisfying the following two conditions is called a \emph{Picard-Vessiot extension}
 for \eqref{LinearSystem}:
\begin{enumerate}
\item[\bf (PV1)]
The field $\Lset$ is generated by $\Kset$
 and elements of a fundamental matrix of \eqref{LinearSystem} ;
\item[\bf (PV2)]
The fields of constants for $\Lset$ and $\Kset$ coincide.
\end{enumerate}
The system \eqref{LinearSystem}
 admits a Picard-Vessiot extension which is unique up to isomorphism.

We now fix a Picard-Vessiot extension $\Lset\supset \Kset$
 and fundamental matrix $\Phi$ with entries in $\Lset$
 for \eqref{LinearSystem}.
Let $\sigma$ be a \emph{$\Kset$-automorphism} of $\Lset$,
 which is a field automorphism of $\Lset$
 that commutes with the derivation of $\Lset$
 and leaves $\Kset$ pointwise fixed.
Obviously, $\sigma(\Phi)$ is also a fundamental matrix of \eqref{LinearSystem}
 and consequently there is a matrix $M_\sigma$ with constant entries
 such that $\sigma(\Phi)=\Phi M_\sigma$.
This relation gives a faithful representation
 of the group of $\Kset$-automorphisms of $\Lset$
 on the general linear group as
\[
R\colon \mathrm{Aut}_{\Kset}(\Lset)\to\GL(n,\mathrm{C}_{\Lset}),
\quad \sigma\mapsto M_{\sigma},
\]
where $\GL(n,\mathrm{C}_{\Lset})$
is the group of $n\times n$ invertible matrices with entries in $\mathrm{C}_{\Lset}$.
The image of $R$
 is a linear algebraic subgroup of $\GL(n,\mathrm{C}_{\Lset})$,
 which is called the \emph{differential Galois group} of \eqref{LinearSystem}
 and often denoted by $\mathrm{Gal}(\Lset/\Kset)$.
This representation is not unique
 and depends on the choice of the fundamental matrix $\Phi$,
 but a different fundamental matrix only gives rise to a conjugated representation.
Thus, the differential Galois group is unique up to conjugation
 as an algebraic subgroup of the general linear group.

Let $G\subset\GL(n,\mathrm{C}_{\Lset})$ be an algebraic group.
Then it contains a unique maximal connected algebraic subgroup $G^0$,
 which is called the \emph{connected component of the identity}
 or \emph{identity component}.
The identity component $G^0\subset G$ is
 the smallest subgroup of finite index, i.e., the quotient group $G/G^0$ is finite.

\section{Monodromy matrices}

In this appendix,
 we give general information on monodromy matrices for the reader's convenience.

Let $\Kset$ be the field of meromorphic functions on a Riemann surface $\Gamma$,
 and consider the linear system \eqref{LinearSystem}.
Let $t_0\in\Gamma$ be a nonsingular point for \eqref{LinearSystem}.
We prolong the fundamental matrix $\Phi(t)$ analytically
 along any loop $\gamma$ based at $t_0$ and containing no singular point,
 and obtain another fundamental matrix $\gamma*\Phi(t)$.
So there exists a constant nonsingular matrix $M_{[\gamma]}$ such that
\[
\gamma*\Phi(t) = \Phi(t)M_{[\gamma]}.
\]
The matrix $M_{[\gamma]}$ depends on the homotopy class $[\gamma]$
 of the loop $\gamma$
 and is called the \emph{monodromy matrix} of $[\gamma]$.

Let $\Lset$ be a Picard-Vessiot extension of \eqref{LinearSystem}
 and let $\mathrm{Gal}(\Lset/\Kset)$ be the differential Galois group,
 as in Appendix~A.
Since analytic continuation commutes with differentiation,
 we have $M_{[\gamma]}\in\mathrm{Gal}(\Lset/\Kset)$.


\end{document}